\documentclass[11pt,oneside,reqno]{amsart}
\usepackage{amsfonts}
\usepackage{amssymb,amsmath, amsthm}
\usepackage[numbers]{natbib}
\usepackage{graphicx}
\usepackage[pdftex,plainpages=false,colorlinks,hyperindex,bookmarksopen,linkcolor=red,citecolor=blue,urlcolor=blue]{hyperref}
\usepackage{mathrsfs}
\usepackage{xcolor}
\usepackage{epstopdf}
\usepackage{color}
\usepackage{enumerate}
\usepackage{comment}
\usepackage{mathabx}

\usepackage[titletoc,toc,title]{appendix}
\usepackage[a4paper,left=2.5cm,right=2.5cm,footskip=1in]{geometry}

\makeatother
\pdfoutput=1
\bibpunct{[}{]}{;}{n}{,}{,}
\newtheorem{theorem}{Theorem}[section]

\newtheorem{coro}{Corollary}[section]
\newtheorem{definition}{Definition}[section]
\theoremstyle{remark}
\newtheorem{remark}{Remark}[section]

\newtheorem{lem}{Lemma}[section]
\newtheorem{example}{Example}[section]

\newcommand{\be}{\begin{eqnarray}}
\newcommand{\ee}{\end{eqnarray}}

\newcommand{\R}{\mathbb{R}}
\newcommand{\C}{\mathbb{C}}
\newcommand{\N}{\mathbb{N}}

\newcommand{\func}[1]{\operatorname{#1}}

\numberwithin{equation}{section}
\allowdisplaybreaks

\begin{document}
\title{A class of infinite-dimensional Gaussian processes defined through generalized fractional operators}


\date{\today }

\author{Luisa Beghin$^{1}$}
\address{${}^1$ Department of Statistical Sciences, Sapienza, University of
Rome. P.le Aldo Moro, 5, Rome, Italy}
\email{luisa.beghin@uniroma1.it}

\author{Lorenzo Cristofaro$^{2}$} 
\address{${}^2$ Department of Statistical Sciences, Sapienza, University of
Rome. P.le Aldo Moro, 5, Rome, Italy}
\email{lorenzo.cristofaro@uniroma1.it}

\author{Yuliya
Mishura$^{3}$}
\address{${}^3$ Department of Mathematics, Taras Shevchenko National
University of Kiev, Volodimirska Street 64, 01033 Kiev, Ukraine}
\email{yuliyamishura@knu.ua}

\begin{abstract}
The generalization of fractional Brownian motion in infinite-dimensional white and grey noise spaces has been recently carried over, following the Mandelbrot-Van Ness representation, through Riemann-Liouville type fractional operators.
Our aim is to extend this construction  by means of general fractional derivatives and integrals, which we define through Bernstein functions. According to the conditions satisfied by the latter, some properties of these processes (such as continuity, local times, variance asymptotics and persistence) are derived. On the other hand, they are proved to display short- or long-range dependence, if obtained by means of a derivative or an integral, respectively, regardless of the
Bernstein function chosen. Moreover, this kind of construction allows us to define the corresponding noise and to derive an Ornstein-Uhlenbeck type process, as solution of an integral equation.
\end{abstract}

\maketitle

\section{Introduction}

Fractional calculus and, in particular, the \emph{Riemann-Liouville fractional operators} (RLFOs hereinafter) were defined for the first time in the XIX century, but only in the last thirty years they have been exploited in a wide range of applications. Very recently, they have been extended to the so-called \emph{generalized fractional operators} (GFOs); see, for example, \cite{KOC}, \cite{CHE} and \cite{LUC}. GFOs consist of integro-differential operators defined by convolution kernels, more general than a power function, with an integrable singularity in the origin. In this framework, a certain class of kernels was identified as being in a Sonine pair related to a certain Bernstein function (see Section \ref{PrelRes}, for details). These operators were used to model anomalous diffusions, subordinated and fractional point processes, in \cite{KOC} and \cite{TOA}, among the others.

In probability theory, besides many other applications, a RLFO was used by Mandelbrot and Van Ness in order to define fractional Brownian motion (FBM) through a Wiener integral with a power-law convolution kernel (see \cite{MAN}). This weighted moving-average representation produces either short- or long-range dependence, according to the power's value of the kernel. This approach is particularly fruitful, since Wiener integrals are widely applied in finance (see, for example, \cite{MIS2}), as well as in other disciplines.

In an originally unrelated framework, so-called \emph{white noise analysis} was developed in order to extend the standard stochastic calculus, based on the standard (Gaussian) white noise, to an infinite-dimensional setting.  The reference texts for this kind of literature are represented by \cite{HID2} and \cite{HID}. The theory of infinite-dimensional spaces allows to define the space of distributions, also called Hida space. Then, by means of the so-called $S$ and $T$ integral transforms, the existence of the noise as distribution (i.e. as the Gateaux derivative of a differentiable generalized random variable) can be proved. These theoretical tools also provide a constructive way to define the noise and allow to study functionals of generalized random variables.

The construction of FBM in infinite-dimensional white noise spaces has been
carried over, through the RLFOs (in \cite{BEN}, \cite{BEN2} and \cite{BIA}) and further generalized to non-Gaussian cases and grey noise spaces, by
many authors (see, for example, \cite{GRO}, \cite{GRO2}, \cite{MUR}, \cite{BOC}, \cite{DAS} and \cite{BEG}).
These new processes exhibit anomalous diffusion properties, self-similarity and a richer covariance structure.  Consequently,  many important functionals, such as Donsker's delta, local times and Feynman-Kac integrals have been defined for them.

Our aim is to extend the construction of FBM in infinite-dimensional white noise spaces, by means of GFOs defined through Bernstein functions. As we will see, the properties of these processes (e.g. continuity, local times, variance asymptotics and persistence) will depend on the conditions satisfied by the corresponding Bernstein function. In particular, we will prove that, regardless of the kernel chosen for the GFO, the corresponding process displays short- or long-range dependence, in the derivative and integral cases, respectively.

Finally, our construction allows us to define the corresponding noise, in a proper distributions' space, and to derive an Ornstein-Uhlenbeck type process, as solution of an integral equation.

The paper is organized as follows: in the next section, we present some preliminary results; these are needed to properly define the GFOs (both in the derivative and integral cases) in terms of tempered distributions, in order to get their Fourier transforms and to define an inner product on the Schwartz space. In particular, we prove some properties of the GFOs, such as the continuity and an integration-by-parts rule. In Section \ref{SecProc}, based on the previous results, we build and characterize the new class of processes (that we will call \emph{Bernstein-Gaussian motion}, hereafter BGM) in the white-noise space; moreover, their distributional properties as well as continuity of sample paths and local time's behavior are analysed. In Section \ref{SecNoise}, we define the corresponding BG-noise as a distribution and compute its $S$-transform; then, we are also able to construct the so-called BG-Ornstein-Uhlenbeck process.
Based on a generalized Mandelbrot-Van Ness representation of the BGM, in Section \ref{SecVar}, we analyze the asymptotic behavior of its variance and covariance functions, giving the conditions under which it displays either short or long range dependence. Finally, in Section \ref{SubSec: SpecialCases}, we show how well-known or new processes, besides FBM, can be obtained as special cases of the BGM, by specifying the Bernstein function. The Appendix contains definitions of well-known fractional operators, as well as some tools of white-noise analysis which are needed for the main results.

\section{Preliminary results}\label{PrelRes}

Let $\varphi:\mathbb{R}^{+}\rightarrow [0,\infty)$ be a Bernstein
function, i.e. a function belonging to $C^{\infty }$ and such that

\begin{equation*}
(-1)^{k-1}\frac{d^{k}}{dx^{k}}\varphi(x)\geq 0,\qquad
x>0,\;k=1,2,\dots
\end{equation*}%
We will restrict ourselves to a complete Bernstein function, i.e. to a $\varphi$ that admits the representation
\begin{equation}
\varphi(x)=a+bx+\int_{0}^{+\infty }(1-e^{-sx})\overline{\nu} (s) ds,\qquad
x>0,  \label{gg}
\end{equation}%
where $\overline{\nu} (\cdot)$ is the (completely monotone) density, w.r.t. the Lebesgue measure, of the L\'{e}vy measure $\Upsilon (\cdot )$ on $(0,+\infty ),$ i.e. $\overline{\nu }%
(z):=\Upsilon (dz)/dz.$
Recall that, by definition, $\int_{0}^{+\infty }(s \wedge 1)\overline{\nu} (s) ds<\infty $ and it is easy to
check that the last condition implies that $\int_{\epsilon}^{+\infty
}\overline{\nu} (s) ds<\infty $,  for any $\epsilon \in (0,1]$. We assume hereafter
that $a=b=0$ in (\ref{gg}).

Let $\nu :\mathbb{R}^{+}\backslash \{0\}\rightarrow \mathbb{R}^{+}$ be the
tail of the L\'{e}vy density, i.e. $\nu (x):=\int_{x}^{+\infty }%
\overline{\nu }(z)dz,$ $x>0$; then $\nu$ is, in general, a right-continuous,
non-increasing function. It is well-known that%
\begin{equation}
\varphi(x)=x\int_{0}^{+\infty }e^{-sx}\nu (s)ds;  \label{tr}
\end{equation}%
see \cite{SCH}. Moreover, we have that
\begin{eqnarray}
\int_{0}^{1}\nu (x)dx &=&\int_{0}^{1}\int_{x}^{+\infty }\overline{\nu }%
(z)dzdx  \label{dd} \\
&=&\int_{0}^{1}z\overline{\nu }(z)dz+\int_{1}^{+\infty }\overline{\nu }%
(z)dz<\infty .  \notag
\end{eqnarray}

Applying Proposition 3.6 in \cite{SCH}, we can extend the Bernstein
function $\varphi(\cdot )$ to the complex half-plane $\overline{\mathcal{H}}%
=\{a+ib:a\geq 0,b\in \mathbb{R}\},$ so that $\varphi(i\xi )$ is
well-defined for all $\xi \in \mathbb{R}.$

We will restrict ourselves to the case where the Bernstein function $\varphi$ satisfies the
following condition:

\begin{enumerate}
\item[$\mathbf{(C)}$] $\lim_{\theta \rightarrow 0}\varphi
(\theta)=0$\, and\, $\lim_{\theta \rightarrow +\infty }\varphi
(\theta)=+\infty.$
\end{enumerate}

Condition $\textbf{(C)}$ implies that $\nu$ is the tail of a L\'{e}vy density with infinite mass on $\mathbb{R}^+$: indeed, by the initial and finite value
theorems, respectively, we have that
\begin{eqnarray}
\lim_{x\rightarrow 0}\nu (x) &=&\lim_{\theta \rightarrow +\infty
}\theta
\tilde{\nu}(\theta )=\lim_{\theta \rightarrow +\infty }\varphi
(\theta)=+\infty,   \label{nu1} \\
\lim_{x\rightarrow +\infty }\nu (x) &=&\lim_{\theta \rightarrow
0}\theta \tilde{\nu}(\theta )=\lim_{\theta \rightarrow 0}\varphi
(\theta)=0,  \label{nu2}
\end{eqnarray}
where we denote by $\tilde{f}(\cdot)$ the Laplace transform of a function $f$.
 Now, let $\kappa :\mathbb{R}^{+}\backslash \{0\}\rightarrow
\mathbb{R}^{+}$ be a locally integrable function (with integrable
singularity at zero) defined as the associate Sonine kernel of $\nu$,
i.e. let $\kappa (\cdot )$ and $\nu (\cdot )$ satisfy the so-called \textquotedblleft Sonine condition"
\begin{equation*}
\int_{0}^{t}\nu (z)\kappa (t-z)dz\equiv 1, \quad t> 0.
\end{equation*}
The latter can be equivalently expressed,
in the Laplace domain, as $\tilde{\kappa}(\theta
)\tilde{\nu}(\theta )=1/\theta $. If $\kappa (\cdot )$ is the
Sonine pair of a tail L\'{e}vy density $\nu
(\cdot )$, by considering (\ref{tr}), we have that%
\begin{equation}
\tilde{\kappa}(\theta )=\int_{0}^{+\infty }e^{-s\theta }\kappa (s)ds=\frac{1%
}{\varphi(\theta )}.  \label{ex}
\end{equation}%
The existence and non-negativity of $\kappa (\cdot )$ is
guaranteed by the
Bernstein theorem, since $\varphi(\cdot )$ is a Bernstein function and thus $%
\tilde{\kappa}(\cdot)$ is completely monotone (by Theorem 3.7 in \cite{SCH}). 
Moreover,  as a consequence of the assumption that the Bernstein function $\varphi(\cdot)$ is complete, we can be sure that $\kappa(\cdot)$ is monotonically non-increasing. 
Indeed, by Theorem 7.3 in \cite{SCH}, the function $1/\varphi(\cdot)$ is Stieltjes and then, by Theorem 2.2 in \cite{SCH}, it is the Laplace transform of a completely monotone density function. 
As a consequence of $\textbf{(C)}$, we have that $\lim_{\theta \rightarrow 0}\tilde{%
\kappa}(\theta )=+\infty$ and $\lim_{\theta \rightarrow +\infty }\tilde{%
\kappa}(\theta )=0$.  For further details on the Sonine theory and on general
fractional calculus in this setting, see \cite{KOC}, \cite{KOC2}, \cite{SAM} and \cite{LUC2}.

\begin{remark}
It is well-known that $1/\varphi(\theta)$ coincides with the Laplace transform of the
potential measure $U(\cdot)$ of the subordinator with Laplace exponent $\varphi$ (see \cite{BER}).
Since, by assumption, $\varphi(\cdot)$ is complete Bernstein, also its conjugate $\varphi^{*}(\theta)=\theta/\varphi(\theta)$ is complete Bernstein (by Prop. 7.1 in \cite{SCH}). Therefore, both $\varphi$ and $\varphi^{*}$ are special Bernstein functions; the related subordinators are, consequently, special and possess non-decreasing potential densities on $(0,+\infty)$. In particular, $U(\cdot)$ has density equal to
\begin{equation}
u(t)=\frac{1}{\int_{0}^{+\infty}x \overline{\nu}(x)dx}+\nu^{*}(t),  \label{star}
\end{equation}
where $\nu^{*}(t):=\int_{t}^{+\infty} \overline{\nu}^{*}(x)dx$ and $\overline{\nu}^{*}(\cdot)$ is the L\'{e}vy density related to $\varphi^{*}(\theta)$
(see \cite{SCH}, p.162).
Thus, when the mean of $\overline{\nu}(\cdot)$ is infinite, the potential density coincides with the tail density and  $\kappa(\cdot)=\nu^{*}(\cdot).$
\end{remark}

\vspace{0.5cm}

We recall now the Young's inequality and we prove some preliminary
results; the latter are needed for defining properly the
GFOs by means of convolutions and
interpreting them as tempered distribution (in order to get its
Fourier transform); the main references are \cite{BOG}, \cite{KAT}
and \cite{RUD}.

\begin{lem}[Young's inequality, \cite{BOG}] 
Let $1\leq p,q,r\leq \infty$ such that $1/p+1/q=1+1/r$; let $f\in L^{p}(\mathbb{R^d})$ and $g\in L^{q}(\mathbb{R^d}%
). $ Then the
function \[(f \star g)(x)=\int_{\mathbb{R}}f(x-y)g(y)dy\] exists, a.e.~$x\in \mathbb{R},$ and $f\star g \in L^r(\R)$.
Moreover, the following inequality holds
\begin{equation*}
\Vert f\star g\Vert _{r}\leq \Vert f\Vert _{p}\Vert g\Vert _{q}.
\end{equation*}
\end{lem}

\begin{lem}
    \label{lemdefintegrator} Let $p_{1}\geq 1$ and $p_{2}\geq 1$ and let $g :%
    \mathbb{R}\rightarrow \mathbb{R^{+}}$ be such that $g(x)=0$, for $x\leq0$, and $g \in
    L^{p_{1}}(0,\epsilon )\cap L^{p_{2}}(\epsilon ,\infty )$, for any $\epsilon
    >0$. Then the operators
    \begin{equation*}
        I_{-}^{(g) }f(x):=(f\star \breve{g})(x),\quad \text{ a.e. }x\in \mathbb{R},
    \end{equation*}
    and
    \begin{equation*}
         I_{+}^{(g)}f(x):=(f\star g)(x),\quad \text{ a.e. }x\in \mathbb{R},
    \end{equation*}%
  where $\breve{g}(y):=g (-y)$,  are  well-defined for $f\in L^{1}(\mathbb{R})$.%
    \newline
\end{lem}

\begin{proof}
   We give the proof only for $I_{-}^{(g)}$, since for $I_{+}^{(g)}$ it is similar.
For $\epsilon >0,$ we define $g _{1}(y):=g (y)1_{(0,\epsilon )}(y)$ and $%
g _{2}(y):=g (y)1_{(\epsilon ,\infty )}(y)$, so that $g _{1}\in L^{m}(\mathbb{R}%
) $ for $m\in \lbrack 1,p_{1}]$ and $g _{2}\in
L^{p_{2}}(\mathbb{R})$. By
the Young's inequality, we have that, for any $f\in L^{1}(\mathbb{R})$, $%
(f\star \breve{g}_{1})(x)$ is well-defined for $x\in
\mathbb{R}/\Omega
_{1} $ and $(f\star \breve{g}_{2})(x)$ is well-defined for $x\in \mathbb{R}%
/\Omega _{2}$ such that $Leb(\Omega _{1})=0$ and $Leb(\Omega _{2})=0$.%
\newline
Moreover, for each $\epsilon >0$ and $x\in \mathbb{R}/(\Omega _{1}\cup
\Omega _{2}),$
\begin{eqnarray*}
|I_{-}^{(g) }f(x)| &\leq&\int_{\mathbb{R}}|f(x-y)|g (-y)dy \\
&=&\int_{-\infty }^{-\epsilon }|f(x-y)|g (-y)dy+\int_{-\epsilon
}^{0}|f(x-y)|g (-y)dy \\
&=&\int_{\mathbb{R}}|f(x-y)|g
_{1}(-y)dy+\int_{\mathbb{R}}|f(x-y)|g
_{2}(-y)dy \\
&=&(|f|\star \breve{g}_{1})(x)+(|f|\star \breve{g}_{2})(x)<\infty
,
\end{eqnarray*}%
by the Young's inequality.\newline
\end{proof}

    \begin{remark}\label{rem:FractIntegralSmooth}
        We highlight that, if we consider the Schwartz space of test functions $\mathcal{S}:=\mathcal{S}(\mathbb{R})$ and $f \in \mathcal{S}$,
        then $I^{(g)}_{\pm}f \in C^{\infty}(\mathbb{R})$, see Theorem 7.19 in \cite{RUD}.
    \end{remark}

\begin{remark}
By the embedding properties of the $L^{p}$ spaces and the hypothesis $%
g _{1}\in L^{p_1}(0,\epsilon )$, we have that $g_1 \in L^{m}(0,\epsilon)$%
, for $m\in \lbrack 1,p_{1}]$. In particular, $g _{1}\in L^{1}(0,\epsilon)$%
. Furthermore, all the examples we give in Section \ref{SubSec:
SpecialCases} are such that, when the function $g$ is the tail of
the L\'{e}vy density, the associated function $g_{2}$ is in
$L^{n}(\epsilon ,\infty ),$
for $n\in \lbrack p_{2},\infty ] $. However, we are assuming that $%
g_2 \in L^{p_2} (\epsilon, \infty)$.

Finally, we note that the hypothesis $g \in L^{p_{1}}(0,\epsilon
)\cap L^{p_{2}}(\epsilon ,\infty )$ for $p_1, p_2 \geq 1$ is
mandatory because $ g $ does not belong necessarily to
$L^{r}(\mathbb{R})$, for any $r\geq 1$ (e.g.
$g(s)=s^{(\alpha-1)/2}/\Gamma((1+\alpha)/2)$, for $\alpha \in
(0,1)$).
\end{remark}

\subsection{The generalized fractional derivative}

We define the generalization of the Riemann-Liouville derivatives, given in (\ref{rl})-(\ref{rl2}), as the
derivatives, in distributional sense, of the operators $I_{-}^{(g) }$ and $I_{+}^{(g) }$.

\begin{definition}
\label{defder} Let $\varphi$ be a Bernstein function and $\nu $ be the tail L\'{e}%
vy density associated to $\varphi ;$ if $\nu :%
    \mathbb{R}\rightarrow \mathbb{R^{+}}$ is such that $\nu(x)=0$, for $x\leq0$, and $\nu \in
    L^{p_{1}}(0,\epsilon )\cap L^{p_{2}}(\epsilon ,\infty )$, for any $\epsilon
    >0$, then, for $f\in
L^{1}(\mathbb{R})$, we define the generalized  right-sided 
Riemann-Liouville fractional derivative as follows
\begin{equation}
(\mathcal{D}_{-}^{(\nu )}f)(x):=-\frac{d}{dx}(I_{-}^{(\nu )}f)(x)=-\frac{d}{%
dx}\int_{x}^{+\infty }f(t)\nu (t-x)dt,\quad \text{a.e. }x\in \mathbb{R},
\label{gg2}
\end{equation}%
 and the generalized left-sided Riemann-Liouville
fractional derivative as follows
    \begin{equation}\label{eqt:LeftSidedRLDerivative}
        (\mathcal{D}_{+}^{(\nu )}f)(x):=\frac{d}{dx}(I_{+}^{(\nu) }f)(x)=\frac{d}{dx}%
        \int_{-\infty}^{x}f(t)\nu (x-t)dt,\quad \text{a.e. }x\in \mathbb{R},
    \end{equation}%
where the differential operator is used in distributional sense.
\newline
\end{definition}

Analogously, the Weyl-type derivatives, given in (\ref{wr})-(\ref{wr2}), can be generalized, to the domain $ W^{1,1}(%
\mathbb{R})$:

\begin{equation*}
W^{1,1}(\mathbb{R})=\{ f \in L^1(\mathbb{R}) | \, Df \in L^1(\mathbb{R}) \},
\end{equation*}
where $D$ is the weak derivative, by the following operators

\begin{equation}
(\mathfrak{D}_{-}^{(\nu )}f)(x):=-(I_{-}^{(\nu) }\frac{d}{dx}%
f)(x)=-\int_{0}^{+\infty }\frac{\partial }{\partial x}f(x+t)\nu (t)dt,\quad
\text{a.e. }x\in \mathbb{R}  \label{weyl}
\end{equation}%
and
    \begin{equation}
        (\mathfrak{D}_{+}^{(\nu )}f)(x):=\left(I_{+}^{(\nu) }\frac{d}{dx}%
        f\right)(x)=\int_{0}^{+\infty }\frac{\partial }{\partial x}f(x-t)\nu (t)dt,\quad
        \text{a.e. }x\in \mathbb{R},  \label{weyl2}
    \end{equation}%
for $f \in W^{1,1}(%
\mathbb{R})$.
The previous definitions have been introduced in \cite{TOA}. For other
references on general fractional calculus with different approaches, see
\cite{CHE}, \cite{LUC}, \cite{KOC} and \cite{KOC2}.

In order to compute the Fourier transform of
$\mathcal{D}_{\pm}^{(\nu )}$, we introduce the definition of the convolution for
distributions, see Section 7 in \cite{RUD}. Let us denote by $\mathcal{S}%
^{\prime }:=\mathcal{S}^{\prime }(\mathbb{R})$ the dual space of
$\mathcal{S}$, i.e. the space of tempered distributions.

    \begin{definition}\label{DefDist}
        Let $f \in L^{p}(\mathbb{R})$ for $p \in [1, +\infty]$. We define the distribution $u_f \in \mathcal{S}'$ as follows
        \[ u_f(q):=\int_{\mathbb{R}}f(y)q(y)dy, \quad q \in \mathcal{S}. \]
        Furthermore, let   $q \in \mathcal{S}$ and $u_f \in \mathcal{S}'$, then their convolution is defined as
        \[  (u_f \star q)(x)=\int_{\mathbb{R}}f(t)q(x-t)dt,\quad  x \in \mathbb{R}, \]
(see Def.~7.8 in \cite{RUD}).\end{definition}

\begin{lem}\label{lem2.3}
Let $p_{1}\geq 1$ and $p_{2}\geq 1$ and let $\nu :%
    \mathbb{R}\rightarrow \mathbb{R^{+}}$ be such that $\nu(x)=0$  for $x\leq0$, and $\nu \in
    L^{p_{1}}(0,\epsilon )\cap L^{p_{2}}(\epsilon ,\infty )$, for any $\epsilon
    >0$. Then the Fourier transform of $%
    \mathcal{D}_{-}^{(\nu )}q$ and $%
    \mathcal{D}_{+}^{(\nu )}q$ are given, respectively, by
    \begin{enumerate}

        \item[i)]
        \begin{equation} \label{eqt:FourierTransformDerivativeLeftSide}
            \mathcal{F}\left( \mathcal{D}_{-}^{(\nu )}q\right) (\xi ):=\int_{\mathbb{R%
            }}e^{i\xi x}\left( \mathcal{D}_{-}^{(\nu )}q\right) (x)dx=\varphi(i\xi )%
            \widehat{q}(\xi ),\qquad \xi \in \mathbb{R},
        \end{equation}
            \item[ii)]
            \begin{equation} \label{eqt:FourierTransformDerivativeRightSide}
                \mathcal{F}\left( \mathcal{D}_{+}^{(\nu )}q\right) (\xi ):=\int_{\mathbb{R%
                }}e^{i\xi x}\left( \mathcal{D}_{+}^{(\nu )}q\right) (x)dx=\overline{\varphi(i\xi )}%
                \widehat{q}(\xi ),\qquad \xi \in \mathbb{R},
            \end{equation}%
where $q \in \mathcal{S}(\mathbb{R})$ and $\widehat{q}(\xi ):=\mathcal{F}\left( q\right) (\xi ).$
    \end{enumerate}
\end{lem}

\begin{proof}
   We prove equation \eqref{eqt:FourierTransformDerivativeLeftSide} in details and we give only some
   hints for equation \eqref{eqt:FourierTransformDerivativeRightSide}.\\
    \begin{itemize}
        \item[i)]
        We first prove the existence of the Fourier transform of the integral
        operator defined in Lemma \ref{lemdefintegrator}; to this aim, we need to
        divide the tail measure into two parts so that each one is in $L^{p}(\mathbb{%
            R})$, for some $p\geq1$. Under these assumptions, we have both the existence of the two parts as
        convolution  and the existence and integral representation of their
        Fourier transform.\newline
Let $\breve{\nu_1}(y):=\nu(-y)1_{(0,\epsilon)}(-y)$ and $\breve{\nu_2}(y):=\nu(-y)1_{[\epsilon, +\infty)}(-y)$; since $\breve{\nu_1} \in L^{p_1}(\mathbb{R})$ and $\breve{\nu_2} \in L^{p_2}(%
        \mathbb{R})$, we can split the operator $I^{(\nu)}_{-}$ into two parts by defining  the tempered distributions as in Def. \ref{DefDist}, so that
        \begin{equation*}
           (u_{\breve{\nu}_i}\star q)(x)=(q \star \breve{\nu}_i)(x),
        \end{equation*}
        for $q \in \mathcal{S}(\mathbb{R})$ and $i=1,2$, see Example 7.12 in \cite{RUD}. Therefore we can write that
        \begin{equation*}
            (I^{(\nu)}_{-}q)(x)=((q \star \breve{\nu}_1)(x) + (q \star \breve{\nu}%
            _2)(x))=((u_{\breve{\nu}_1} + u_{\breve{\nu}_2}) \star q)(x).
        \end{equation*}
        By Theorem 7.19 in \cite{RUD}, we have that if $q \in \mathcal{S}(\mathbb{%
            R})$, then $I^{(\nu)}_{-}q \in C^{\infty}(\mathbb{R})$ and
        \begin{equation*}
            \mathcal{F}(I^{(\nu)}_{-}q)(\xi)=\int_{\mathbb{R}}e^{i\xi x}((u_{\breve{\nu%
                }_1} + u_{\breve{\nu}_2}) \star q)(x)dx=\int_{\mathbb{R}}e^{i\xi x}(
                \breve{\nu}_1 + \breve{\nu}_2)(x)dx\int_{\mathbb{R}}e^{i\xi
                x}q(x)dx.
        \end{equation*}
Moreover, under the assumptions on $\nu(\cdot)$, we have that $\nu_1 \in L^{1}(\mathbb{R})$ and $\nu_2\in L^{p_2}(\R)\cap L^{\infty}(\R)$, where we recall that $\nu_{1}(y):=\nu (y)1_{(0,\epsilon )}(y)$ and $%
\nu _{2}(y):=\nu (y)1_{(\epsilon ,\infty )}(y)$, $\epsilon >0$.  Thus, both $u_{\breve{\nu}_1}$ and $u_{\breve{\nu}_2}$
        have an integral representation on $\mathcal{S}(\mathbb{R})$ and we are
        able to compute explicitly the Fourier transform of the distributions $u_{%
            \breve{\nu}_1}$ and $u_{\breve{\nu}_2}$, see Ch.~VI, Sect.~4 in \cite{KAT}.
            
       Their Fourier transforms exist
        singularly:

        \begin{eqnarray*}
            \int_{\mathbb{R}}e^{i\xi x}(\breve{\nu}_1 + \breve{\nu}_2)(x)dx&=&
            \int_{\mathbb{R}}e^{i\xi x}\breve{\nu}_1(x)dx + \int_{\mathbb{R}}e^{i\xi x}%
            \breve{\nu}_2(x)dx \\
            &\overset{*}{=}&\int_0^\infty e^{-i\xi y}\nu(y)dy \\
            &\overset{**}{=}&\frac{\varphi(i \xi)}{i \xi},
        \end{eqnarray*}
        for $\xi \neq 0$, since $\nu_1(\cdot) + \nu_2(\cdot)=\nu(\cdot)$ and we have used the change of variable $y=-x$ in $*$ and equation~\eqref{tr} in $**$. 
        Formula (\ref{eqt:FourierTransformDerivativeLeftSide}) follows by
        considering Def. \ref{defder} and
        recalling that $\mathcal{F}\left(\frac{d}{dx}f \right)(\xi)=(-i\xi)\mathcal{F%
        }\left(f \right)(\xi)$.

        \item[ii)] In order to prove \eqref{eqt:FourierTransformDerivativeRightSide}, we have that the existence follows from
        similar considerations as above and the following lines are needed:

            \begin{eqnarray*}
                \int_{\mathbb{R}}e^{i\xi x}(\nu_1 + \nu_2)(x)dx&=&
                \int_{\mathbb{R}}e^{i\xi x}\nu_1(x)dx + \int_{\mathbb{R}}e^{i\xi x}%
                \nu_2(x)dx \\
                &=&\int_0^\infty e^{i\xi x}\nu(x)dx \\
                &\overset{*}{=}&\overline{\left(\frac{\varphi(i \xi)}{i \xi}\right)},
            \end{eqnarray*}
            where we use ~\eqref{tr} in $*$.

    \end{itemize}
\end{proof}
If we choose $\varphi(x)=x^{\beta },$ $\beta \in (0,1)$, the generalized Riemann-Liouville derivatives defined in Def.\ref{defder}
coincide with the fractional derivatives of order $\beta$ (given in
(\ref{rl})-(\ref{rl2})); in this special case, the tail L\'{e}vy density is given by $\nu (s)=s^{-\beta }/\Gamma (1-\beta )$ and $\mathcal{F}\left( \mathcal{D}%
_{\pm}^{(\nu )}f\right) (\xi )=(\mp i\xi )^{\beta }\widehat{f}(\xi ).$

If $\nu \in L^1_{loc}(\mathbb{R})$, its Fourier
transform exists in the distributional sense, but the associated
distribution is defined on $\mathcal{D}(\mathbb{R})$, where
$\mathcal{D}(\mathbb{R})=\{ q \in C^{\infty}(\mathbb{R})| \,
Supp(q) \text{ is compact}\}$.

\begin{remark}\label{Rem}
We note that, on the Schwartz space $\mathcal{S}(\mathbb{R})$, the operators $\mathcal{%
D}_{\pm}^{(\nu )}$ coincide with the Weyl-type derivatives $\mathfrak{D}%
_{\pm}^{(\nu )}$ defined in (\ref{weyl})-(\ref{weyl2}), as can be checked by comparing (\ref%
{eqt:FourierTransformDerivativeLeftSide})-(\ref%
{eqt:FourierTransformDerivativeRightSide}) with Lemma 2.9 in
\cite{TOA}.
\end{remark}

\

We now derive the general fractional derivative of the indicator
function and we prove that it belongs to
$L^{2}(\mathbb{R})$. To this aim, in view of what follows we restrict ourselves to the right-sided case.

Even though the definition of $\mathcal{D}_{-}^{(\nu )}$ is given on $L^{1}(%
\mathbb{R})$, its Fourier transform, given in formula (\ref{eqt:FourierTransformDerivativeLeftSide}), holds for $%
f\in \mathcal{S}(\mathbb{R})$. We extend it to the indicator
function in distributional sense. We will denote hereafter the positive part as usual, i.e. 
\begin{equation*}
(x)_{+}:=\left\{
\begin{array}{c}
x,\quad x>0 \\
0,\quad x\leq 0%
\end{array}%
\right . .
\end{equation*}

\begin{theorem}
\label{thmFourierconvolution} Let $a,b\in \mathbb{R}$ such that $a<b$, then
the generalized fractional derivative of the indicator function of the
interval $(a,b)$ is given by
\begin{equation}
(\mathcal{D}_{-}^{(\nu )}1_{(a,b)})(x)=\nu ((b-x)_{+})-\nu ((a-x)_{+}),
\quad \text{a.e.}\, x\in \mathbb{R}.  \label{gg3}
\end{equation}%
Moreover, we have that
\begin{equation}
\mathcal{F}\left( \mathcal{D}_{-}^{(\nu )}1_{(a,b)}\right) (\xi
)=\varphi(i\xi ) \widehat{1_{(a,b)}}(\xi )=\varphi(i\xi
)\frac{e^{ib\xi }-e^{ia\xi }}{i\xi }, \label{ff}
\end{equation}%
in distributional sense.
\end{theorem}

\begin{proof}
Thanks to Remark \ref{Rem} and considering that $%
1_{(a,b)} \in W^{1,1}(\mathbb{R})$, for $a,b \in \mathbb{R}$ and $a<b$, we can compute directly the Weyl-type derivative of the indicator function. Since $\frac{d}{dx}1_{(a,b)}(x)=\delta_a(x)-%
\delta_b(x)$, a.e.~$x$, we have that
\begin{eqnarray*}
(\mathfrak{D}_{-}^{(\nu )}1_{(a,b)})(x)&=&-\int_{0}^{+\infty }\frac{\partial
}{\partial x}1_{(a,b)}(x+t)\nu (t)dt \\
&=&-\int_{0}^{+\infty }(\delta_a(x+t)-\delta_b(x+t))\nu (t)dt \\
&=&\nu ((b-x)_{+})-\nu ((a-x)_{+}), \qquad \text{a.e. }
x\in \mathbb{R}.
\end{eqnarray*}
It could be checked that this is equivalent to using the
Riemann-Liouville derivative for $1_{(a,b)}$ by taking a sequence
$\{q_n\} \subset \mathcal{S}$ such that $q_n \to 1_{(a,b)}$
in $L^1(\mathbb{R})$.

Let $\langle \cdot ,\cdot \rangle_{L^2}$ be the inner product of
$L^{2}$. Then $\langle \mathcal{F}(\mathcal{D}_{-}^{(\nu
)}1_{(a,b)})(\cdot),\eta \rangle_{L^2}=\langle \varphi(i\cdot
)\mathcal{F}(1_{(a,b)}),\eta \rangle_{L^2} $, for $\eta \in
\mathcal{S}$. Indeed, by the Plancherel's theorem,
\begin{eqnarray*}
\langle \mathcal{F}(\mathcal{D}_{-}^{(\nu )}1_{(a,b)}),\eta \rangle_{L^2} &=&\int_{%
\mathbb{R}}\nu ((b-x)_{+})\mathcal{F}(\eta
)(x)dx-\int_{\mathbb{R}}\nu
((a-x)_{+})\mathcal{F}(\eta )(x)dx \\
&=&\int_{-\infty }^{b}\nu ((b-x))\mathcal{F}(\eta
)(x)dx-\int_{-\infty
}^{a}\nu ((a-x))\mathcal{F}(\eta )(x)dx \\
&=&\int_{0}^{\infty }\nu
(y)\int_{\mathbb{R}}(e^{it(b-y)}-e^{it(a-y)})\eta
(t)dtdy \\
&=&\int_{\mathbb{R}}(e^{itb}-e^{ita})\frac{\varphi(it)}{it}\eta (t)dt \\
&=&\int_{\mathbb{R}}\varphi(it)\mathcal{F}(1_{a,b})(t)\eta (t)dt,
\end{eqnarray*}%
since  $(\mathcal{F}1_{(a,b)})(t)=(e^{ibt}-e^{iat})/it$.
\end{proof}

\begin{theorem}
\label{L2} Let $a,b\in \mathbb{R}$ be such that $a<b$, and let $\overline{\nu }(\cdot )$ and $\nu (\cdot )$
satisfy the following assumptions:%
\begin{enumerate}
\item[$\mathbf{(A1)}$] $\overline{\nu} \in L^{2}(1 ,+\infty )$.

\item[$\mathbf{(A2)}$] $\nu \in L^{2}(0,1)\cap L^{p}(1 ,+\infty )$, for
some $p\geq 1.$
\end{enumerate}%
Then $\mathcal{D}_{-}^{(\nu )}1_{(a,b)}\in L^{2}(\mathbb{R}).$
\end{theorem}

\begin{proof}
In order to prove that $\int_{\mathbb{R}}\left[ (\mathcal{D}_{-}^{(\nu
)}1_{(a,b)})(x)\right] ^{2}dx<\infty $, we can write that
\begin{eqnarray*}
\int_{-\infty }^{a-1}\left[ \nu (b-u)-\nu (a-u)\right] ^{2}du
&=&\int_{-\infty }^{a-1}\left[ \int_{a-u}^{b-u}\overline{\nu}(w)dw\right] ^{2}du \\
&=&\int_{0}^{\epsilon}\left[ \int_{z+1}^{b-a+z+1}\overline{\nu}(w)dw\right] ^{2}dz+\int_{\epsilon }^{+\infty}\left[ \int_{z+1}^{b-a+z+1}\overline{\nu}(w)dw\right] ^{2}dz  \\
&\leq& \int_{0}^{\epsilon}\left[ \int_{1}^{b-a+1}\overline{\nu}(w)dw\right] ^{2}dz+(b-a)^2 \int_{\epsilon }^{+\infty}\overline{\nu}(z+1)^{2} dz <\infty ,
\end{eqnarray*}%
by the complete monotonicity of $\overline{\nu}(\cdot)$ and by \textbf{(A1)}. Let now $\delta >0$ be such that $a+\delta <b,$ and let $\epsilon
:=b-a-\delta$, then, since $b-x>\epsilon $ and the tail measure is
non-increasing, we have that $\nu (b-x)^{2}\leq \nu (\epsilon )^{2}<\infty $
and we can write%
\begin{eqnarray*}
&&\int_{a-1}^{a+\delta }\left[ \nu ((b-u)_{+})-\nu ((a-u)_{+})\right] ^{2}du
\\
&\leq &\int_{a-1}^{a+\delta }\nu (b-u)^{2}du+\int_{a-1}^{a}\nu (a-u)^{2}du \\
&\leq &(1+\delta )\nu (\epsilon )^{2}+\int_{0}^{1}\nu (z)^{2}dz<\infty ,
\end{eqnarray*}%
by $\textbf{(A2)}$. Finally,

\begin{equation*}
    \int_{a+\delta}^{b} \nu (b-u)^2 du= \int_{0}^{\epsilon} \nu(x)^2 dx <\infty.
\end{equation*}
\end{proof}

We note that the hypothesis $\textbf{(A2)}$ is a requirement on $\nu $ stronger than $\nu \in L^{p}(0,\epsilon )$, with $p\geq 1$, given in Lemma \ref%
{lemdefintegrator}.\\

\subsection{The generalized fractional integral}

By the considerations of Remark \ref{remder}, in order to recover
the case of the fractional Brownian motion with Hurst parameter
$H=\alpha /2>1/2$, we now explore the situation where the process
is defined by means of a
general fractional integral, instead of the derivative $\mathcal{D}%
_{-}^{(\nu )}.$

Let us recall the definition of $\kappa (\cdot)$ given in
(\ref{ex}), for a certain Bernstein function $\varphi$, and let us assume that $\kappa(\cdot) \in L^{1}(0,\epsilon]$, for any $\epsilon>0$. 
Then we
need also to define here the function $\chi
(x):=\int_{0}^{x}\kappa (z)dz,$
for $x>0$; its Laplace transform is equal to $\widetilde{\chi }(\theta )=%
\tilde{\kappa}(\theta )/\theta =1/\theta \varphi(\theta
)=1/\tilde{\nu}(\theta
)\theta ^{2}.$

In this case, we assume the further condition on the Bernstein function $\varphi(\cdot)$ (and thus on the corresponding Sonine pair)
\begin{enumerate}
\item[$\mathbf{(K)}$] $\lim_{\theta \rightarrow
0}\varphi(\theta )/\theta=\infty $ and $\lim_{\theta \rightarrow
+\infty }\varphi(\theta )/\theta=0$.
\end{enumerate}

As a consequence of $\textbf{(K)}$ we have the following asymptotic behavior of
the function $\kappa (\cdot )$:%
\begin{eqnarray}
\lim_{x\rightarrow 0}\kappa (x) &=&\lim_{\theta \rightarrow
+\infty }\theta
\tilde{\kappa}(\theta )=+\infty   \label{kappa1} \\
\lim_{x\rightarrow +\infty }\kappa (x) &=&\lim_{\theta \rightarrow
0}\theta \tilde{\kappa}(\theta )=0.  \label{kappa2}
\end{eqnarray}%
Finally, under $\textbf{(C)-(K)}$, we have that%
\begin{eqnarray}
\lim_{x\rightarrow 0}\chi (x) &=&\lim_{\theta \rightarrow +\infty }\theta
\widetilde{\chi }(\theta )=0 \\
\lim_{x\rightarrow +\infty }\chi (x) &=&\lim_{\theta \rightarrow 0}\theta
\widetilde{\chi }(\theta )=+\infty, \label{chilim}
\end{eqnarray}
while $\chi (\cdot)$ is increasing, by definition.

We consider the generalized Riemann-Liouville integral, by
choosing the kernel
$\kappa (\cdot )$ in the definition of the operator $\mathcal{I}_{-}^{(g)}$ given in Lemma \ref%
{lemdefintegrator} and evaluate its Fourier transform, in distributional
sense, for $f\in L^{1}(\mathbb{R})$.

\begin{definition}
\label{defint} Let $\kappa (\cdot)$ be a function defined in (\ref{ex}), for a certain Bernstein
function $\varphi$. If  $\kappa \in L^{p_{1}}(0,\epsilon )\cap L^{p_{2}}(\epsilon ,+\infty )$,
for any $\epsilon >0$ and for some $p_{1},p_{2}\geq 1$, then, for $f\in L^{1}(%
\mathbb{R})$, we define the generalized right-sided Riemann-Liouville
fractional integral as follows
\begin{equation}
(\mathcal{I}_{-}^{(\kappa )}f)(x):=-\int_{x}^{+\infty }f(t)\kappa
(t-x)dt,\quad \text{a.e. }x\in \mathbb{R}  \label{ex2}
\end{equation}
and the generalized left-sided Riemann-Liouville
fractional integral as follows
    \begin{equation}\label{eqt:LeftSidedRLIntegral}
        (\mathcal{I}_{+}^{(\kappa )}f)(x):=
        \int_{-\infty}^{x}f(t)\kappa (x-t)dt,\quad \text{a.e. }x\in \mathbb{R}.
    \end{equation}%
\end{definition}
Note that, in view of (\ref{star}), the previous operator could be alternatively defined as 
$(\mathcal{I}_{-}^{(\kappa )}f)(x):=\int_{x}^{+\infty }f(t)\nu^{*}
(t-x)dt,$ a.e. $ x\in \mathbb{R}$.

Thanks to Lemma \ref{lemdefintegrator}, the
operator $\mathcal{I}_{-}^{(\kappa )}$ introduced in
Def.~\ref{defint} is well-defined and its Fourier transform is
given by
\begin{equation}
\mathcal{F}\left( \mathcal{I}_{-}^{(\kappa )}\eta \right) (\xi ):=\int_{%
\mathbb{R}}e^{i\xi x}\left( \mathcal{I}_{-}^{(\kappa )}\eta
\right) (x)dx=\varphi(i\xi )^{-1}\widehat{\eta }(\xi ),\qquad \xi
\in \mathbb{R}, \label{ex3}
\end{equation}%
for $\eta \in \mathcal{S}(\mathbb{R})$ and $\widehat{\eta }(\xi ):=\mathcal{F%
}\left( \eta \right) (\xi ).$ Formula (\ref{ex3}) can be obtained following the same reasoning of Lemma \ref{lem2.3} and taking into account equation (\ref{ex}).

\begin{remark}
If we choose, in particular, $\kappa (s)=s^{\beta -1}/\Gamma (\beta ),$
for $\beta \in (0,1),$ which corresponds to $\nu (s)=s^{-\beta }/\Gamma
(1-\beta )$, formula (\ref{ex2}) reduces to the (right-sided)
Riemann-Liouville integral
\begin{equation}
(\mathcal{I}_{-}^{\beta }\eta )(x)=\frac{1}{\Gamma (\beta )}%
\int_{x}^{+\infty }\eta (t)(t-x)^{\beta -1}dt.  \label{ex4}
\end{equation}%
Moreover, in this case, formula (\ref{ex3}) gives the well-known
Fourier transform of (\ref{ex4}): $\mathcal{F}\left(
\mathcal{I}_{-}^{\beta }\eta \right) (\xi )$ $=(i\xi )^{-\beta
}\widehat{\eta }(\xi )$ (see \cite{KIL}, p.90).
\end{remark}

We now apply $\mathcal{I}_{-}^{(\kappa )}$ to the indicator
function.

\begin{theorem}\label{kap}
Let the functions $\chi(\cdot)$ and $\kappa(\cdot)$ satisfy the following conditions,
for any $\varepsilon
>0$,

\begin{enumerate}
\item[$\mathbf{(B1)}$] $\kappa \in L^{1}(0 ,\varepsilon) \cap L^{2}(\varepsilon ,+\infty )$,

\item[$\mathbf{(B2)}$] $\chi \in L^{2}(0,\varepsilon )$.
\end{enumerate}
Then
\begin{equation}
(\mathcal{I}_{-}^{(\kappa )}1_{(a,b)})(x)=\chi ((b-x)_{+})-\chi
((a-x)_{+}),\quad \text{a.e. }x\in \mathbb{R},  \label{nor}
\end{equation}%
with Fourier transform, in distributional sense,
\begin{equation}
\mathcal{F}\left( \mathcal{I}_{-}^{(\kappa )}1_{(a,b)}\right) (\xi
)=\varphi(i\xi )^{-1}\frac{e^{ib\xi }-e^{ia\xi }}{i\xi },
\label{fint}
\end{equation}%
and $\mathcal{I}_{-}^{(\kappa )}1_{(a,b)}\in L^{2}(\mathbb{R}).$
\end{theorem}

\begin{proof}
Formula (\ref{nor}) is obtained directly, by applying
Def.\ref{defint}. The Fourier transform (\ref{fint}) follows
immediately from (\ref{ex3}). In order to prove that
$\mathcal{I}_{-}^{(\kappa )}1_{(a,b)}\in L^{2}(\mathbb{R})$ we write
\begin{eqnarray*}
&&\int_{-\infty }^{+\infty }\left[ \chi ((b-x)_{+})-\chi ((a-x)_{+})\right]
^{2}dx \\
&=&\int_{-\infty }^{a}\left[ \chi (b-x)-\chi (a-x)\right] ^{2}dx+%
\int_{a}^{b}\chi (b-x)^{2}dx.
\end{eqnarray*}%
The second integral is proved to be finite, by \textbf{(B2)}, since
\begin{equation*}
\int_{a}^{b}\chi (b-x)^{2}dx=\int_{0}^{b-a}\chi (w)^{2}dw<\infty ,
\end{equation*}%
while, for the first one, we can write (by considering that $\kappa (\cdot )$ is continuous and non-increasing)%
\begin{eqnarray*}
\int_{-\infty }^{a}\left[ \chi (b-x)-\chi (a-x)\right] ^{2}dx
&=&\int_{-\infty }^{a}\left[ \int_{a-x}^{b-x}\kappa (z)dz\right] ^{2}dx \\
&=&\int_{0}^{+\epsilon }\left[ \int_{w}^{w+b-a}\kappa (z)dz\right] ^{2}dw+\int_{\epsilon}^{+\infty }\left[ \int_{w}^{w+b-a}\kappa (z)dz\right] ^{2}dw  \\
&&\leq \int_{0}^{+\epsilon }\left[ \int_{0}^{b-a}\kappa (z)dz\right] ^{2}dw+(b-a)^2 \int_{\epsilon}^{+\infty }\kappa (z)^{2}dw <\infty ,
\end{eqnarray*}%
by \textbf{(B1)}.
\end{proof}

\subsection{The inner products induced by the Bernstein function}
Let now consider the nuclear triple $\mathcal{S}\subset L^{2}(\mathbb{R}%
)\subset \mathcal{S^{\prime }}$, then we define the following inner products on $\mathcal{S}$.

\begin{definition}\label{Inner}
Let $\varphi$ be a Bernstein function satisfying $\mathbf{(C)}$-$\mathbf{(K)}$, then, for any $\xi, \eta \in \mathcal{S}$, we define
\begin{equation*}
\langle \eta ,\xi \rangle _{\nu }:=C_{\nu }\int_{\mathbb{R}}|\varphi(ix)|^{2}%
\overline{\hat{\xi}(x)}\hat{\eta}(x)dx,
\end{equation*}%
where $\nu (\cdot )$ is the tail L\'{e}vy measure corresponding to
the Bernstein function $\varphi$ and $C_{\nu
}:=(||1_{[0,1)}||_{\nu }^{2})^{-1}=\left( \int_{\mathbb{R}}\left[
\nu ((1-u)_{+})-\nu ((-u)_{+})\right] ^{2}du\right) ^{-1}$. Similarly, we define

\begin{equation*}
\langle \eta ,\xi \rangle _{\kappa }:=C_{\kappa }\int_{\mathbb{R}}|\varphi(ix)|^{-2}%
\overline{\hat{\xi}(x)}\hat{\eta}(x)dx,
\end{equation*}%
where $C_{\kappa }:=(||1_{[0,1)}||_{\kappa }^{2})^{-1}$
$=\left( \int_{\mathbb{R}}\left[ \chi ((1-u)_{+})-\chi
((-u)_{+})\right] ^{2}du\right) ^{-1}$ and $\kappa
(\cdot )$ is the associate Sonine pair. 
\end{definition}

 It is
easy to see that both $\langle \cdot ,\cdot \rangle _{\nu }$ and  $\langle \cdot ,\cdot \rangle _{\kappa }$ are
well-defined, since the symmetry and linearity are
immediate consequences of the definitions. They are Hermitian,
since, for any $\xi,\eta \in \mathcal{S}$,
\begin{eqnarray*}
\langle \xi ,\eta \rangle _{\nu } &=&C_{\nu }\int_{\mathbb{R}}|\varphi(ix)|^{2}%
\overline{\hat{\xi}(x)}\hat{\eta}(x)dx \\
&=& C_{\nu }\overline{\int_{\mathbb{R}}|\varphi(ix)|^{2}\overline{\hat{\eta}(x)}%
\hat{\xi}(x)dx} =\overline{\langle \eta ,\xi \rangle _{\nu }}
\end{eqnarray*}
and the same holds for $\langle \cdot ,\cdot \rangle _{\kappa }$.
It is also evident that $\langle \xi ,\xi \rangle _{\nu
}\geq0$, $\langle \xi ,\xi \rangle _{\kappa
}\geq0$, for any $\xi,$  and that they vanish if and only if $\xi$ is
equal to zero a.e..

 We show now that the following relationships hold for the inner products in $L^{2}(\mathbb{R})$:

\begin{equation}
\langle \eta ,\xi \rangle _{\nu }= C_\nu \langle \mathcal{D}_{-}^{(\nu )}\eta ,%
\mathcal{D}_{-}^{(\nu )}\xi \rangle _{L^{2}} \quad \text{and} \quad \langle \eta ,\xi \rangle _{\kappa }= C_\kappa\langle \mathcal{I}_{-}^{(\kappa )}\eta ,%
\mathcal{I}_{-}^{(\kappa )}\xi \rangle _{L^{2}}, \label{rel}
\end{equation}%
for $\eta ,\xi \in \mathcal{S}(\mathbb{R}).$ Indeed, in the first case, we have that

\begin{eqnarray}
\langle \eta ,\xi \rangle _{\nu } &=&C_{\nu}\int_{\mathbb{R}}|\varphi(ix)|^{2}%
\overline{\hat{\xi}(x)}\hat{\eta}(x)dx=C_{\nu}\int_{\mathbb{R}}\overline{%
\varphi(ix)\hat{\xi}(x)}\varphi(ix)\hat{\eta}(x)dx \label{up} \\
&=&C_{\nu}\int_{\mathbb{R}}\overline{\widehat{\mathcal{D}_{-}^{(\nu )}\xi (x)}}%
\widehat{\mathcal{D}_{-}^{(\nu )}\eta (x)}dx \notag \\
&=&C_{\nu}\langle \widehat{\mathcal{D}_{-}^{(\nu )}\eta },\widehat{\mathcal{D}%
_{-}^{(\nu )}\xi }\rangle _{L^{2}}=C_{\nu}\langle \mathcal{D}_{-}^{(\nu )}\eta ,%
\mathcal{D}_{-}^{(\nu )}\xi \rangle _{L^{2}},\notag
\end{eqnarray}%
where, in the last step, we have used the Plancherel's theorem. The second case follows similarly.

We know, from Theorem \ref{L2} and Theorem \ref{kap}, that both $\mathcal{D}_{-}^{(\nu
)}1_{(a,b)}$ and $\mathcal{I}_{-}^{(\kappa
)}1_{(a,b)}$ belong to $L^2 (\mathbb{R})$ and then the equivalences given in (\ref%
{rel}) hold, not only in $\mathcal{S}(\mathbb{R})$, but also for the
indicator functions.
Thus, by Theorem \ref{thmFourierconvolution}, we have that
\begin{eqnarray}
\langle 1_{[0,t)},1_{[0,s)}\rangle _{\nu }  &=&C_{\nu }\int_{\mathbb{R}}\varphi (ix)\widehat{1_{[0,t)}}(x)%
\overline{\varphi (ix)\widehat{1_{[0,s)}}(x)}dx \label{eq} \\
&=&C_{\nu }\int_{\mathbb{R}}\frac{\left\vert \varphi(ix)\right\vert ^{2}}{x^{2}}%
\left[ (e^{-itx}-1)(e^{isx}-1)\right] dx \notag
\\
&=&\int_{\mathbb{R}}\mathcal{F}(\left[ \nu (t-\cdot)_{+}-\nu (-\cdot)_{+}%
\right])(x) \overline{\mathcal{F}(\left[ \nu (s-\cdot)_{+}-\nu (-\cdot)_{+}%
\right])}(x)dx \notag \\
&=&\text{[by the Plancherel's theorem]} \notag\\
&=&\int_{\mathbb{R}}\left[ \nu (t-u)_{+}-\nu (-u)_{+}%
\right]  \left[ \nu (s-u)_{+}-\nu (-u)_{+}%
\right] du \notag \\
&=&C_{\nu }\langle\mathcal{D}^{(\nu)}_{-}1_{[0,t)},\mathcal{D}^{(\nu)}_{-}1_{[0,s)}\rangle_{L^2} < +\infty, \notag
\end{eqnarray}
Analogously, by Theorem \ref{kap}, we can write that
\begin{eqnarray}
\langle 1_{[0,t)},1_{[0,s)}\rangle _{\kappa }&=&C_{\kappa }\int_{\mathbb{R}}\frac{(e^{-itx}-1)(e^{isx}-1)}{x^{2}\left\vert \varphi(ix)\right\vert^{2}}%
dx  \label{eqint} \\
&=&\int_{\mathbb{R}}\left[ \chi (t-u)_{+}-\chi (-u)_{+}\right]  \left[ \chi (s-u)_{+}-\chi (-u)_{+}%
\right] du \notag \\
&=&C_{\kappa }\langle \mathcal{I}_{-}^{(\kappa
)}1_{[0,t)},\mathcal{I}_{-}^{(\kappa )}1_{[0,s)}\rangle _{L^{2}}< +\infty. \notag
\end{eqnarray}

\begin{remark}\label{remder}
If we put $\varphi(x)=x^{\frac{1-\alpha }{2}}$, for $\alpha \in
(0,1)$, with
corresponding tail L\'{e}vy measure $\nu (x)=x^{\frac{\alpha -1}{2}%
}/\Gamma ((1+\alpha )/2)$, we obtain the case of the fractional
Brownian motion with Hurst parameter $H=\alpha/2$, for $H<1/2$:
indeed we have that $\mathcal{D}_{-}^{(\nu
)}=\mathcal{D}_{-}^{(1-\alpha)/2}$. Moreover, equation
(\ref{rel}) coincides with the result given in \cite{GRO2}, Lemma
3.3, i.e.
\begin{equation*}
\langle \eta ,\xi \rangle _{\alpha }=C_{\alpha}\langle
\mathcal{D}_{-}^{(1-\alpha )/2}\eta ,\mathcal{D}_{-}^{(1-\alpha
)/2}\xi \rangle _{L^{2}},
\end{equation*}%
where $\langle \eta ,\xi \rangle _{\alpha }:=C_{\alpha }\int_{\mathbb{R}%
}|x|^{1-\alpha }\overline{\hat{\xi}(x)}\hat{\eta}(x)dx$, for $\eta
,\xi \in \mathcal{S}(\mathbb{R})$ and $C_{\alpha}=\sin (\pi \alpha/2) \Gamma (1+\alpha)$.

We note that the restriction on the parameter $\alpha$ is needed since, otherwise, i.e. for $\alpha >1$,
the function $\nu(\cdot)$ would not be a non-increasing function, as a proper tail L\'{e}vy density. Therefore, this case must be generalized by considering the general fractional integral,
instead of the derivative.

\end{remark}
In order to simplify the notation and to recover more easily the well-known special case of
the FBM, we give the following
\begin{definition} \label{DefOperatorM} For a complete Bernstein function $\varphi
(\cdot)$ satisfying $\mathbf{(C)}$-$\mathbf{(K)}$, let $\mathcal{M}_{\pm}^{(\varphi )}$ be defined either as
$\mathcal{M}_{\pm}^{(\nu )}:=C_{\nu}\mathcal{D}_{\pm}^{(\nu )}$, under
the assumptions \textbf{(A1)-(A2)}, or as $\mathcal{M}_{\pm}^{(\kappa )}:=C_{\kappa}\mathcal{I}_{\pm}^{(\kappa )}$, under \textbf{(B1)-(B2)}. 
\end{definition}

\subsection{Properties of the GFOs} \label{sec:AnalyticalPropertiesLeftSided}

Finally, we prove the continuity of the right-sided GFO $\mathcal{M}_{+}^{(\varphi )}$ defined in Def. \ref{DefOperatorM} and the integration by parts rule for a special set of functions.
As in \cite{BEN}, the above properties of continuity and integration by parts play a central role for the BG-noise's existence.  

	\subsubsection*{Continuity}
	We introduce the complexified Schwartz space as
	
	\[  \mathcal{S}_{\mathbb{C}}:=\{ \xi_1 + \mathrm{i}\xi_2 \mid \xi_i \in \mathcal{S} \text{ for }i=1,2\} \]
	and the families $(|\cdot|_p)_{p \in \mathbb{N}}$ and $(|\cdot|_{r,s})_{r,s \in \mathbb{N}}$ defined as follows, for $\xi \in \mathcal{S}$,
	\[  |\xi|_{r,s}:=\sup_{x \in \mathbb{R}}|x^r \frac{d^s}{dx^s} \xi(x)|  \]
	and  $\| \cdot \|_p$ are defined in Appendix~\ref{App:HidaDistribution}.
	The previous two families of norms and semi-norms are equivalent on $\mathcal{S}$, see Section V.3 in \cite{REE}.
	\begin{lem}\label{lem:ContinuityFracOpera}
		Let $p_{1,\nu},p_{2,\nu},p_{1,\kappa},p_{2,\kappa} \geq 1$, $\nu \in L^{p_{1,\nu}}((0,\epsilon))\cap L^{p_{2,\nu}}([\epsilon,\infty))$ and $\kappa \in L^{p_{1,\kappa}}((0,\epsilon))\cap L^{p_{2,\kappa}}([\epsilon,\infty))$ for any $\epsilon>0$. If $\varphi$ is either $\kappa$ or $\nu$, then
		
		\[  \mathcal{M}_{+}^{\varphi}: (\mathcal{S}_{\mathbb{C}},\|\cdot\|_p) \to (L^{\infty}(\mathbb{R}),\|\cdot\|_{L^\infty(\R)})  \]
		is bounded. \\
		Specifically, there exist $p_\varphi \in \mathbb{N}$ and $C_\varphi>0$ such that $\|\mathcal{M}_{+}^{\varphi}\xi\|_{L^\infty(\R)}< C_{\varphi} \| \xi \|_{p_{\varphi}}$, for any $\xi \in \mathcal{S}_{\mathbb{C}}$. 
	\end{lem}
	
	\begin{proof}
		Let $\xi \in \mathcal{S}_{\mathbb{C}}$ with $\xi_1 = \Re(\xi)$ and $\xi_2 = \Im(\xi)$. By noting that $\mathcal{M}_{+}^{\varphi}\xi=\mathcal{M}_{+}^{\varphi}\xi_1+\mathrm{i}\mathcal{M}_{+}^{\varphi}\xi_2$, we firstly prove the statement for $\mathcal{M}_{+}^{\varphi}=\mathcal{M}_{+}^{\nu}$ and then for $\mathcal{M}_{+}^{\varphi}=\mathcal{M}_{+}^{\kappa}$ with $\xi \in \mathcal{S}$ as for $\xi \in \mathcal{S}_{\mathbb{C}}$ is straightforward. \\
		Let $\mathcal{M}_{+}^{\varphi}=\mathcal{M}_{+}^{\nu}$, $\xi \in \mathcal{S}$, $x \in \mathbb{R}$ and for any $\epsilon>0$, we split the L{\'e}vy measure as $\nu(s)=\nu_1(s)+\nu_2(s)$ for $s>0$ where $\nu_1(s)=\nu(s)1_{(0,\epsilon)}(s)$ and $\nu_2(s)=\nu(s)1_{(\epsilon,\infty)}(s)$. We have that
		
		\begin{eqnarray*}
			(\sqrt{C_\nu})^{-1}\left|\mathcal{M}_{+}^{\nu}\xi(x)\right|&=&\left|\frac{d}{dx} \int_{-\infty}^x \xi(s)\nu(x-s)ds \right|\\
			&=& \left|\frac{d}{dx} \int_{0}^\infty \xi(x-s)\nu(s)ds \right|\\
			&\leq&  \int_{0}^\infty \left|\frac{d}{dx}\xi(x-s)\right|\nu_1(s)ds + \frac{d}{dx} \int_{0}^\infty \left|\xi(x-s)\right|\nu_2(s)ds\\
			&\leq& \max_{x \in \mathbb{R}}\left|\xi'(x)\right|\|\nu_1\|_{L^{1}(\mathbb{R}^+)}+\left|\frac{d}{dx} \int_{0}^\infty \xi(x-s)\nu_2(s)ds\right|.
		\end{eqnarray*}
		
		For the last term, we have that
		
		\[
		\int_{0}^\infty \xi(x-s)\nu_2(s)ds=\int_{\epsilon}^\infty \xi(x-s)\int_s^ \infty \bar{\nu}(y)dyds
		=\int_\epsilon^\infty \bar{\nu}(y)\int_\epsilon^y \xi(x-s)dsdy,
		\]
		and 
		\[ \left|\frac{d}{dx} \int_\epsilon^y\xi(x-s)ds\right|=\left|\frac{d}{dx} \int_{x-y}^{x-\epsilon}\xi(u)du\right|<|\xi(x-\epsilon)|+|\xi(x-y)|\leq 2 \max_{x \in \mathbb{R}}|\xi(x)|, \]
		so that we obtain the following inequality
		
		\begin{eqnarray*}  
			\left|\frac{d}{dx} \int_{0}^\infty \xi(x-s)\nu_2(s)ds\right|&=& \int_\epsilon^\infty \bar{\nu}(y)dy\left|\int_\epsilon^y\frac{d}{dx} \xi(x-s)ds\right|\\
			&<& 2 \max_{x \in \mathbb{R}}|\xi(x)|\int_\epsilon^\infty \bar{\nu}(y)dy\\
			&=&2 \max_{x \in \mathbb{R}}|\xi(x)| \nu(\epsilon).
		\end{eqnarray*}
		Thus, 
		\[ 	|\mathcal{D}_{+}^{\nu}\xi(x)|< \max_{x \in \mathbb{R}}|\xi'(x)|\|\nu_1\|_{L^{1}(\mathbb{R}^+)}+2 \max_{x \in \mathbb{R}}|\xi(x)| \nu(\epsilon)<C_{\nu, \epsilon}(|\xi|_{0,1}+|\xi|_{0,0}) \]
		where $C_{\nu, \epsilon}=\|\nu_1\|_{L^{1}(\mathbb{R})}+2\nu(\epsilon)$.
		
		

         Finally, by Lemma 2 in \cite{REE}, p.~142,   
         , we finally have that there exists $p' \in \mathbb{N}$ such that $\|\mathcal{D}_{+}^{\nu}\xi\|_{\infty}< C \| \xi \|_{p'}$,
		for $C>0$.\\
        Let now $\mathcal{M}_{+}^{\varphi}=\mathcal{M}_{+}^{\kappa}$, $\xi \in \mathcal{S}$, $x \in \mathbb{R}$ and for any $\epsilon>0$, we split $\kappa$ as $\nu$ above such that

        \begin{eqnarray*}
            (\sqrt{C_\nu})^{-1}\left|\mathcal{M}_{+}^{\kappa}\xi(x)\right|&=&\left| \int_{-\infty}^x \xi(s)\kappa(x-s)ds \right|\\
            &\leq& \int_0^\infty |\xi(x-s)|\kappa(s)ds=\int_0^\infty |\xi(x-s)|\kappa_1(s)ds + \int_0^\infty |\xi(x-s)| \kappa_2(s)ds\\
            &\leq& \|\xi\|_{0,0} \|\kappa_1\|_{L^1(\R)} +\kappa(\epsilon) \|\xi\|_{L^1(\R)}.
        \end{eqnarray*}
        From p.~135 in \cite{REE}, it follows that $\|\xi \|_{L^1(\mathbb{R})}\leq \pi( \|\xi \|_{L^\infty(\R)} +\|x^2 \xi \|_{L^\infty(\R)})$. Using Lemma 2 in p.~142 of \cite{REE} and leveraging the fact that $|\cdot|_{r,s}$ are directed, we obtain $\|\mathcal{M}_{+}^{\kappa}\xi \|_{L^\infty(\R)}\leq C_{\kappa,\epsilon}|\xi |_{p''}$, where $C_{\kappa,\epsilon}=\|\kappa_1\|_{L^1(\R)}+\kappa(\epsilon)\pi$ and $p'' \in \mathbb{N}$.\\
        For $\epsilon=1$ we have that $\|\mathcal{M}_{+}^{\varphi}\xi \|_{L^\infty(\R)}\leq C_\varphi \| \xi \|_{p_\varphi} $ where $p_\varphi \in \N$, so that the result follows. 
		
	\end{proof}
	
	\subsubsection*{Integration by parts rule} We prove the following result, which is needed below for the definition of the BG-noise.

\begin{lem}\label{lem:IntegrPartsFormul}
		Let $\xi \in \mathcal{S}$ and let $\mathcal{M}_{-}^{(\varphi )}$ be defined as in Def. \ref{DefOperatorM}; if assumptions \textbf{(A1)-(A2)} and \textbf{(B1)-(B2)} hold, then 
		\begin{equation} \label{eq: IntegPartsIndica}
			\langle \xi, \mathcal{M}_{-}^{(\varphi )} 1_{(0,t]} \rangle_{L^2(\R)}= \langle \mathcal{M}_{+}^{(\varphi )} \xi , 1_{(0,t]} \rangle_{L^2(\R)}. 
		\end{equation}
	\end{lem}

 \begin{proof}
 
    Let $\xi \in \mathcal{S}$ and $\mathcal{M}_{-}^{(\varphi )}\xi=\mathcal{M}_{-}^{(\nu )}\xi$. By noting that $\mathcal{M}_{-}^{(\nu )} 1_{(0,t]}\in L^2(\R)$ by Theorem~\ref{L2} and also that $\mathcal{M}_{+}^{(\nu )} \xi=\sqrt{C_\nu}(\nu_1\star D\xi) + \sqrt{C_\nu}(\nu_2\star D\xi) \in L^2(\R)$ by the Young's inequality, where $D$ is the derivative, we apply the Parseval's formula in $*$ as follows: 

    \begin{eqnarray*}
			\langle \xi, \mathcal{M}_{-}^{(\nu )} 1_{(0,t]} \rangle_{L^2(\R)}&\overset{*}{=}& 	\langle \overline{\hat{\xi}}, \sqrt{C_\nu} g(\mathrm{i}\cdot) \widehat{1_{(0,t]}} \rangle_{L^2(\R)}\\
			&=&\langle \overline{\sqrt{C_\nu} \overline{g(\mathrm{i} \cdot) \widehat{\xi} }}, \widehat{1_{[0,t)]}}\rangle_{L^2(\R)}\\
			&=& \langle \overline{ \widehat{\mathcal{M}_{+}^{(\nu )}\xi)}}, \widehat{1_{[0,t)]}}\rangle_{L^2(\R)}\\
			&\overset{*}{=}& \langle \mathcal{M}_{+}^{(\nu )}\xi, 1_{[0,t)}\rangle_{L^2(\R)}.
		\end{eqnarray*}

    For $\mathcal{M}_{-}^{(\varphi )}\xi=\mathcal{M}_{-}^{(\kappa )}\xi$ the result follows analogously.

\end{proof}

\section{The BG-process in the white noise space}\label{SecProc}

Let us now consider the space $\mathcal{S}^{\prime }$ with its cylindrical
sigma-field $\mathcal{B}$ and let us indicate by $\langle \cdot, \cdot
\rangle$ the dual pairing between $\mathcal{S}$ and $\mathcal{S}^{\prime }$,
defined as an extension of the inner product in $L^2$. In particular, for $%
f\in L^2$ and $\eta \in \mathcal{S}$, we have that $\langle f,
\eta
\rangle=\langle f, \eta \rangle_{L^2}$. We endow $(\mathcal{S}^{\prime },%
\mathcal{B})$ with the white noise measure $\mu $ defined by

\begin{equation}
\int_{\mathcal{S}^{\prime }(\mathbb{R})}e^{i\langle u,\xi \rangle }d\mu
(u)=\exp\left( -\frac{1}{2}\langle \xi ,\xi \rangle_{L^2} \right) ,
\label{def}
\end{equation}%
for $\xi \in \mathcal{S}(\mathbb{R})$. We recall that, thanks to the
properties of the white noise measure $\mu $ (see \cite{OBA}), the following
equalities  hold  for the moments:
\begin{eqnarray}
&&\int_{\mathcal{S}^{\prime }(\mathbb{R})}\langle u,\xi \rangle ^{2p}d\mu
(u) =\frac{(2p)!}{2^{p}p!}\langle \xi ,\xi \rangle_{L^2} ^{p}, \qquad \int_{%
\mathcal{S}^{\prime }(\mathbb{R})}\langle u,\xi \rangle ^{2p+1}d\mu (u) =0,
\label{ss} \\
&&\int_{\mathcal{S}^{\prime }(\mathbb{R})}\langle u,\xi \rangle
\langle u,\eta \rangle d\mu (u) =\langle \xi ,\eta \rangle_{L^2} ,
\label{ss2}
\end{eqnarray}%
for $p\in \mathbb{N}$, $\xi ,\eta \in \mathcal{S}(\mathbb{R}).$ The above equations are extended to the complexified space of square integrable functions $L^2_{\C}$ as well as the following:

\begin{equation}\label{eq:ExpComplexSchwartz}
 \int_{\mathcal{S}^{\prime }(\mathbb{R})} e^{\langle u, \xi \rangle} d\mu (u)=e^{\langle \xi,\xi \rangle/2}, \quad \xi \in L^2_{\C},  \end{equation}
see p.~21 in \cite{OBA}.
We now consider the extension of the dual pairing $\langle \cdot,
\cdot \rangle$ to $\mathcal{S}^{\prime}(\mathbb{R})\times
L^2(\mathbb{R},dx)$ and we recall that, by Theorem \ref{L2},
$\mathcal{D}_{-}^{(\nu )}1_{[0,t)} \in L^2(\mathbb{R},dx)$;
analogously, we proved in Theorem \ref{kap} that
$\mathcal{I}_{-}^{(\kappa )}1_{[0,t)} \in L^2(\mathbb{R},dx)$.

\begin{definition} Let $\varphi
(\cdot)$ be a complete Bernstein function satisfying $\mathbf{(C)}$-$\mathbf{(K)}$ and let $\mathcal{M}_{-}^{(\varphi )}$ be the generalized operator given in Def. \ref{DefOperatorM}
. Then we define on
$(\mathcal{S}^{\prime },\mathcal{B},\mu )$ the
BG-process as%
\begin{equation}
X_{\varphi }(1_{[0,t)}):=\langle \omega ,\mathcal{M}_{-}^{(\varphi
)}1_{[0,t)}\rangle,\qquad t\in \mathbb{R}^{+},\text{ }\text{a.s. }
\omega \in \mathcal{S}^{\prime }(\mathbb{R}).  \label{def4}
\end{equation}
\end{definition}

We now present some distributional properties of
$X_{\varphi }(1_{[0,t)})$ and we prove that, under
appropriate conditions on the Bernstein function, it has a
continuous version and a square-integrable local time.

\begin{theorem}
\label{T1} {Let

\begin{equation*} C_{\varphi}:=\left\{
\begin{array}{l}
C_{\nu},\text{\qquad if }\mathcal{M}_{-}^{(\varphi )}=\mathcal{M}%
_{-}^{(\nu )} \\
C_{\kappa},\text{\qquad if }\mathcal{M}%
_{-}^{(\varphi )}=\mathcal{M}_{-}^{(\kappa )}%
\end{array}%
\right.
\end{equation*}
and let
\begin{equation*} A_{\varphi}(x):=\left\{
\begin{array}{l}
\varphi_1(ix)^{2}+\varphi_2(ix)^{2},\text{\qquad if }\mathcal{M}_{-}^{(\varphi )}=\mathcal{M}%
_{-}^{(\nu )} \\
\left[ \varphi_1(ix)^{2}+\varphi_2(ix)^{2}\right] ^{-1},\text{\qquad if }\mathcal{M}%
_{-}^{(\varphi )}=\mathcal{M}_{-}^{(\kappa )}%
\end{array}%
\right.
\end{equation*}
where $\varphi_1(z):=\mathfrak{R}(\varphi(z))$ and $\varphi_2(z):=\mathfrak{I}(\varphi(z))$, for $z \in \mathbb{C}$,} then for the process $X_{\varphi
}(1_{[0,t)})$ defined in (\ref{def4}) we have that

\begin{enumerate}
\item[$(i)$] It is Gaussian, with zero mean and
\begin{equation*}
\mathrm{Cov}(X_{\varphi }(1_{[0,t)}),X_{\varphi }(1_{[0,s)}))=2C_{\varphi }\int_{\mathbb{%
R}}A_{\varphi}(x)\frac{ 1-\cos (tx)-\cos (sx)+\cos
((t-s)x)}{x^{2}} dx.
\end{equation*}

\item[$(ii)$] For $0<t_{1}...<t_{k},$ $k\in \mathbb{N},$ and $\theta _{j}\in
\mathbb{R},$ $j=1,2,,...k,$ its characteristic function reads
\begin{equation*}
\mathbb{E}e^{i\sum_{j=1}^{k}\theta _{j}X_{\varphi }(1_{[0,t_{j})})}=\exp\left( -%
\frac{1}{2}\left\Vert \sum_{j=1}^{k}\theta
_{j}1_{[0,t_{j})}\right\Vert _{\varphi }^{2}\right),
\end{equation*}
where $\left\Vert \xi\right\Vert_{\varphi }:=\sqrt{\langle \xi,\xi
\rangle _{\varphi }}$ and $\langle \cdot,\cdot \rangle _{\varphi
}$ is either $\langle \cdot,\cdot \rangle _{\nu }$ or $\langle
\cdot,\cdot \rangle _{\kappa }$, defined in Def. \ref{Inner}.

\item[$(iii)$] It has stationary increments

\item[$(iv)$] If the Bernstein function $\varphi(\cdot )$ is such
that the following
condition is satisfied:%
\begin{equation}
\int\limits_{\mathbb{R}}\frac{A_{\varphi}(ax)}{x^{2}}(1-\cos x)dx\leq Ka^{1-%
\frac{q+1}{p}},  \label{ass}
\end{equation}%
for any $a>0$ and some constants $K,q>0$, $p>1,$ with $q<p-1$, then there
exists a version of the process with $\gamma $-Holder continuous sample
paths, for $0<\gamma <q/2p.$
\end{enumerate}
\end{theorem}

\begin{proof}
\begin{enumerate}
\item[$(i)$] The values of mean and covariance follow by considering (\ref{ss}) and (%
\ref{ss2}) respectively: indeed, for  the case where $\mathcal{M}_{-}^{(\varphi )}=\mathcal{M}%
_{-}^{(\nu )}$ we have that

\begin{eqnarray}
&&\mathrm{Cov}(X_{\varphi }(1_{[0,t)}),X_{\varphi }(1_{[0,s)}))  \label{cov2} \\
&=&\langle \mathcal{M}_{-}^{(\nu )}1_{[0,t)},\mathcal{M}_{-}^{(\nu
)}1_{[0,s)}\rangle_{L^2}+\overline{\langle \mathcal{M}_{-}^{(\nu )}1_{[0,t)},%
\mathcal{M}_{-}^{(\nu )}1_{[0,s)}\rangle_{L^2}}  \notag \\
&=&\text{by (\ref{eq})}=\langle 1_{[0,t)},1_{[0,s)}\rangle _{\nu }+\overline{%
\langle 1_{[0,t)},1_{[0,s)}\rangle _{\nu }}  \notag \\
&=&2C_{\nu }\int_{\mathbb{R}}\frac{\varphi_1(ix)^{2}+\varphi_2(ix)^{2}}{x^{2}}\left[ 1-\cos
(tx)-\cos (sx)+\cos ((t-s)x)\right] dx.  \notag
\end{eqnarray}
The case where $\mathcal{M}_{-}^{(\varphi )}=\mathcal{M}%
_{-}^{(\kappa )}$ can be treated by analogous steps.

\item[$(ii)$] The formula for the characteristic function easily follows from the linearity
of definition (\ref{def4}).

\item[$(iii)$] We can check that $X_{\varphi }(1_{[0,t)})$ has
stationary increments, since, for $h>0,$
\begin{eqnarray}
&&\mathbb{E}e^{i\sum_{j=1}^{k}\theta _{j}\left[ X_{\varphi
}(1_{[0,t_{j}+h)})-X_{\varphi }(1_{[0,h)})\right] }  \label{cf} \\
&=&\exp\left( -\frac{1}{2}\left\Vert \sum_{j=1}^{k}\theta
_{j}(1_{[0,t_{j}+h)}-1_{[0,h)})\right\Vert _{\varphi }^{2}\right)  \notag \\
&=&\exp\left( -\frac{1}{2}\left\Vert \sum_{j=1}^{k}\theta
_{j}1_{[h,t_{j}+h)}\right\Vert _{\varphi }^{2}\right) =\exp\left( -\frac{1}{2}%
\left\Vert \sum_{j=1}^{k}\theta _{j}1_{[0,t_{j})}\right\Vert
_{\varphi }^{2}\right) .  \notag
\end{eqnarray}%
The last step follows by considering that $\mathcal{F}\left(1_{[h,t_{j}+h)}%
\right) (x)=(e^{ix(t_{j}+h)}-e^{ixh})/ix=e^{ixh}(e^{ixt_{j}}-1)/ix$ and that
\begin{eqnarray*}
\| 1_{[h,t_j+h)}\|_{\varphi}^2&=& C_{\varphi }\int_{\mathbb{R}}A_{\varphi}(x)\overline{%
\frac{e^{ix(t_{j}+h)}-e^{ixh}}{ix}}\frac{e^{ix(t_{j}+h)}-e^{ixh}}{ix}dx \\
&=& C_{\varphi }\int_{\mathbb{R}}A_{\varphi}(x)e^{-ixh} \frac{e^{-ixt_{j}}-1}{ix}%
e^{ixh} \frac{e^{ixt_{j}}-1}{ix}dx=\|1_{[0,t_j)} \|_{\varphi}^2.
\end{eqnarray*}

\item[$(iv)$] In order to apply the Kolmogorov continuity theorem, it is
enough to prove that the following inequality is satisfied, under the
assumption (\ref{ass}):%
\begin{equation}
\mathbb{E}\left[ \left( X_{\varphi }(1_{[0,t)})-X_{\varphi }(1_{[0,s)})\right) ^{2p}%
\right] \leq K^{\prime }(t-s)^{q+1},  \label{ass2}
\end{equation}%
for some $K^{\prime }>0$ and $s<t$. We can write that%
\begin{eqnarray*}
\mathbb{E}\left[ \left\vert X_{\varphi }(1_{[0,t)})-X_{\varphi
}(1_{[0,s)})\right\vert ^{2p}\right] &=&\int_{\mathcal{S}^{\prime }(\mathbb{R%
})}\langle u,\mathcal{M}_{-}^{(\varphi )}1_{[s,t)}\rangle ^{2p}d\mu (u) \\
&=&[\text{by (\ref{ss})]}=\frac{(2p)!}{2^{p}p!}\langle
\mathcal{M}_{-}^{(\varphi
)}1_{[s,t)},\mathcal{M}_{-}^{(\varphi )}1_{[s,t)}\rangle^{p} \\
&=&[\text{by (\ref{up})]}=C_{\varphi
}^{p}\frac{(2p)!}{2^{p}p!}\left[
\int_{\mathbb{R}}A_{\varphi}(x)\left\vert (\mathcal{F}1_{[s,t)})(x)\right\vert ^{2}dx%
\right] ^{p} \\
&=&C_{\varphi }^{p}\frac{(2p)!}{2^{p}p!}\left[
\int_{\mathbb{R}}A_{\varphi}(x)\left\vert (\mathcal{F}1_{[0,t-s)})(x)\right\vert ^{2}dx%
\right] ^{p}.
\end{eqnarray*}%
Then we obtain that%
\begin{eqnarray}
&&\mathbb{E}\left[ \left\vert X_{\varphi }(1_{[0,t)})-X_{\varphi
}(1_{[0,s)})\right\vert ^{2p}\right]  \label{dis} \\
&=&C_{\varphi }^{p}\frac{(2p)!}{p!}\left[
\int_{\mathbb{R}}\frac{A_{\varphi}(x)}{x^{2}}\left[ 1-\cos
((t-s)x)\right] dx\right] ^{p}
\notag \\
&=&C_{\varphi }^{p}\frac{(2p)!(t-s)^{p}}{p!}\left[ \int_{\mathbb{R}}\frac{%
A_{\varphi}(z/(t-s))}{z^{2}}\left[ 1-\cos z\right] dz%
\right] ^{p}  \notag \\
&\leq &K^{\prime }(t-s)^{q+1},  \notag
\end{eqnarray}%
by (\ref{ass}).
\end{enumerate}
\end{proof}

We now derive the condition on the Bernstein function under which
the process $X_{\varphi }(1_{[0,t)})$ possesses a continuous and
square integrable local time. We recall that a sufficient
condition for these properties is that, for $0<s<t\leq T<+\infty
$, there
exist $(\rho _{0},H)\in (0,+\infty )\times (0,1)$ and a positive function $%
\psi \in L^{1}(\mathbb{R})$ such that the following upper bound holds for
the characteristic function of a measurable process $\{X(t)\}_{t\in (0,T)}$
\begin{equation}
\left\vert \mathbb{E}\left( \exp \left\{ i\theta \frac{X(t)-X(s)}{(t-s)^{H}}%
\right\} \right) \right\vert \leq \psi (\theta ),  \label{bon}
\end{equation}%
for $|t-s|<\rho _{0}$ and $\theta \in \mathbb{R}$ (see \cite{BON} for
details). Let us denote by $L(B,x)$, for the Borel set $B\subseteq \mathbb{R}%
^{+},$ the Radon-Nikodym derivative w.r.t.~the Lebesgue measure of $\mu
_{B}(A)=\int_{B}1_{\{X(s)\in A\}}ds,$ for the Borel set $A\subseteq \mathbb{R%
}$.

\begin{theorem}
\label{T2} If the Bernstein function $\varphi$ satisfies the following condition%
\begin{equation}
\int\limits_{\mathbb{R}}\frac{A_{\varphi}(xa)}{x^{2}}(1-\cos
x)dx\geq Ca^{\beta },  \label{bon2}
\end{equation}%
for any $a>0$ and some constant $\beta \in (0,1)$, then
$X_{\varphi }(1_{[0,t)})$ possesses local time $L([0,t],x)$
continuous in $t$, for $a.e.$ $x\in \mathbb{R}$, and square
integrable with respect to $x$.
\end{theorem}

\begin{proof}
From (\ref{cf}) we have that
\begin{eqnarray*}
&&\left\vert \mathbb{E}\left(\exp \left\{ i\theta \frac{X_{\varphi
}(1_{[0,t)})-X_{\varphi
}(1_{[0,s)})}{(t-s)^{(1-\beta )/2}}\right\} \right) \right\vert \\
&=&\exp \left\{ -\frac{\theta ^{2}}{2|t-s|^{1-\beta }}\left\Vert
1_{[0,|t-s|)}\right\Vert _{\varphi }^{2}\right\} \\
&=&\exp \left\{ -\frac{\theta ^{2}C_{\varphi }|t-s|}{2|t-s|^{1-\beta }}\int_{%
\mathbb{R}}\frac{A_{\varphi}(z/|t-s|)}{z^{2}}\left[ 1-\cos
z\right] dz\right\} \\
&\leq &\exp \left\{ -\theta ^{2}C^{\prime }\right\} :=\psi (\theta ),
\end{eqnarray*}%
for a constant $C^{\prime }>0$, where, in the last step, we have applied (%
\ref{bon2}). Thus the $L([0,t],x)$ is continuous in $t$, for $a.e.$ $x\in
\mathbb{R}$, and square integrable with respect to $x$, by (\ref{bon}),
since $\psi (\theta )\in L^{1}(\mathbb{R})$.
\end{proof}
\begin{remark}
As a particular case, for $\varphi(x)=x^{(1-\alpha )/2}$ and $\alpha \in (0,1)$, the
process $X_{\varphi }(1_{[0,t)})$ coincides with the fractional
Brownian motion $B_{t}^{H}$, with Hurst parameter
$H=\alpha /2<1/2$ (see \cite{MUR} for details on the special cases),
since $(ix)^{(1-\alpha )/2}=|x|^{(1-\alpha)\mathrm{sign}(x)/4}$ and thus $\varphi_1(x)^{2}+\varphi_2(x)^{2}=|x|^{1-\alpha }$.
On the one hand, in this case, the condition (\ref{ass}) is satisfied with
$(q+1)/p=\alpha $, for
any $p$ such that $\alpha p>1$, since we have that%
\begin{equation}
\int_{\mathbb{R}}\frac{(ax)^{1-\alpha }}{x^{2}}\left[ 1-\cos x\right]
dx=Ka^{1-\alpha },  \label{fbm}
\end{equation}%
and $\mathbb{E}\left[ \left\vert X_{\varphi
}(1_{[0,t)})-X_{\varphi }(1_{[0,s)})\right\vert ^{2p}\right] \leq
K^{\prime }(t-s)^{\alpha p}.$ Thus we obtain the well-known result
of the $\gamma $-H\"{o}lder continuity of the fractional Brownian
motion, with $0<\gamma <H$, since $\gamma =q/(2p)=(\alpha
/2)-1/(2p)$. Moreover, for $\alpha =1$, it agrees with the
well-known H\"{o}lder property of the Brownian motion.

On the other hand, the condition (\ref{bon2}), which is satisfied for $\beta
=1-\alpha $, in view of (\ref{fbm}), guarantees the continuity and square
integrability of the local time for the fractional Brownian motion, which is
already well-known (see \cite{BON} and \cite{AZM}).
\end{remark}

    \section{The BG-noise}\label{SecNoise}

    In order to define the noise as a distribution, we need to introduce the space of distributions
     $(\mathcal{S})^{-1}$. In the Appendix B, we recall some well-known properties and the construction
     of spaces in white noise analysis, in order to give the definition of $S$-transform. Based on these and on the
     characterization theorems, we introduce here the BG-noise and we compute its $S$-transform.
     Then, we are also able to define the BG-Ornstein-Uhlenbeck process. 

    \begin{theorem}\label{thm:NoiseExistence}
        Let the assumptions \textbf{(A1)-(A2)} hold on $\nu$ and  \textbf{(B1)-(B2)} hold for $\kappa$, then the generalized process $(0,\infty) \ni t \mapsto X_{\varphi}(1_{[0,t)}) \in (\mathcal{S})^{-1}$ is differentiable, i.e. the following limit

        \[   \lim_{h \to 0} \frac{X_{\varphi}(1_{[0,t+h)}) - X_{\varphi}(1_{[0,t)})}{h}, \quad t>0 \]
        exists and converges in $(\mathcal{S})^{-1}$ to the element denoted by $\mathscr{N}_t^{\varphi}$, which fulfils 
\[  S (\mathscr{N}_t^{\varphi})(\xi)=(\mathcal{M}_{+}^{(\varphi )}\xi)(t), \]
        for every $\xi$ in a suitable neighborhood of zero $U_{p',q}=\{\xi \in \mathcal{S}_{\mathbb{C}} \mid 2^q \|\xi\|_{p'}^2<1 \}\subset \mathcal{S}_{\mathbb{C}}$, where $p',q \in \mathbb{N}$.
    \end{theorem}
    \begin{proof}
        Let $t>0$ and define the elements of the sequence $\{ \mathscr{N}_{t,n}^{\varphi}\}_{n \in \mathbb{N}}$ by

        \[  \mathscr{N}_{t,n}^{\varphi}:= \frac{X_{\varphi}(1_{[0,t+h_n)}) - X_{\varphi}(1_{[0,t)})}{h_n}, \quad n \in \mathbb{N},  \]
        where $\{ h_n\}_{n \in \mathbb{N}}$ is a sequence of real numbers such that $h_n \to 0$, for $n \to \infty$. By Theorem~\ref{thm:CharctConvergSequenceHidaSpace}, the sequence $\{ \mathscr{N}_{t,n}^{\varphi}\}_{n \in \mathbb{N}}$ is in $ (\mathcal{S})^{-1}$ and converges strongly in $(\mathcal{S})^{-1}$. Indeed, in the following step (1), we compute the S-transform of each $\mathscr{N}_{t,n}^{\varphi}$, on $U_{p,q}\supset U_{p^{'},q}$, with $p<p^{'}$. In order to apply Theorem~\ref{thm:CharctConvergSequenceHidaSpace}, we prove that $S(\mathscr{N}_{t,n}^{\varphi})(\xi)$ is a Cauchy sequence for all $\xi \in U_{p,q}$ and that $S(\mathscr{N}_{t,n}^{\varphi})(\xi)$ is bounded and holomorphic on $U_{p^{'},q}$, in steps (2) and (3), respectively.

        \begin{enumerate}

            \item For $\omega \in \mathcal{S}'$ and $s \in [-1,1]$, we define the function
            $f(\omega,s):=\exp(\langle\omega, \xi + s \mathcal{M}_{-}^{(\varphi )}1_{[0,t)} \rangle)$
             which is in $L^1(\mu)$, by Equation~\eqref{eq:ExpComplexSchwartz}, and is differentiable w.r.t.~$s$:
            \[ \frac{d}{d s} f(\omega, s)=\langle \omega, \mathcal{M}_{-}^{(\varphi )}1_{[0,t)}\rangle e^{\langle\omega, \xi + s \mathcal{M}_{-}^{(\varphi )}1_{[0,t)} \rangle}.\]
            For each $s \in [-1,1]$ and $\omega \in \mathcal{S}'$, we have that $|\frac{d}{d s} f(\omega, s)|<g(\omega)$, where
            \[  g(\omega)= e^{\langle \omega, \xi_1 \rangle} e^{2|\langle \omega, \mathcal{M}_{-}^{(\varphi )} 1_{[0,t)}\rangle|}, \]
            and $\xi_1=\Re(\xi)$.\\
            By the H\"{o}lder's inequality we have that $g \in L^1(\mu)$, because the functions
            $ e^{\langle \omega, \xi_1 \rangle}$ and $e^{2|\langle \omega, \mathcal{M}_{-}^{(\varphi )} 1_{[0,t)}\rangle|}$
            are in $L^2(\mu)$, by Equation~\eqref{eq:ExpComplexSchwartz}. Hence, for any $\xi \in U_{p,q}$, we
            have that

            \begin{eqnarray*}
                S(X_{\varphi}(1_{[0,t)}))(\xi)&=&e^{-\frac{1}{2}\langle \xi, \xi \rangle}\int_{\mathcal{S}'}\langle \omega ,\mathcal{M}_{-}^{(\varphi )}1_{[0,t)}\rangle e^{\langle \omega, \xi\rangle } \mu(d\omega)\\
                &=&e^{-\frac{1}{2}\langle \xi, \xi \rangle} \int_{\mathcal{S}'} \frac{d}{d s} e^{\langle \omega, \xi\rangle + s\langle \omega,\mathcal{M}_{-}^{(\varphi )}1_{[0,t)}  \rangle}\Big|_{s=0} \mu(d\omega)\\
                &\overset{*}{=}&e^{-\frac{1}{2}\langle \xi, \xi \rangle} \frac{d}{d s}\int_{\mathcal{S}'}  e^{\langle \omega, \xi + sM_{-}^{(\varphi)}1_{[0,t)}  \rangle} \mu(d\omega)\Big|_{s=0}\\
                &\overset{**}{=}&e^{-\frac{1}{2}\langle \xi, \xi \rangle} \frac{d}{d s} e^{\frac{1}{2} \langle \xi + s\mathcal{M}_{-}^{(\varphi )}1_{[0,t)}, \xi + s\mathcal{M}_{-}^{(\varphi)}1_{[0,t)} \rangle }\Big|_{s=0}\\
                &=&\langle \xi, \mathcal{M}_{-}^{(\varphi )}1_{[0,t)}\rangle
            \end{eqnarray*}
by interchanging the derivative and integral in $*$, as $g \in L^1(\mu)$, and by using again Equation~\eqref{eq:ExpComplexSchwartz} in $**$.

            Finally, we have that

            \begin{equation}\label{eq:STransformBrown}
            S(X_{\varphi}(1_{[0,t)}))(\xi)=\langle \xi, \mathcal{M}_{-}^{(\varphi )}1_{[0,t)}\rangle=\int_0^t (\mathcal{M}_{+}^{(\varphi )}\xi) (x)dx, \end{equation}
            where, in the last equality, we use Lemma \ref{lem:IntegrPartsFormul} and the linearity of the operator $\mathcal{M}_{+}^{(\varphi )}$.

            \item   
            We apply the $S$-transform to $\mathscr{N}_{t,n}^{\varphi}$ with $\xi \in U_{p,q}$ :

             \begin{eqnarray*}
                S\big(\mathscr{N}_{t,n}^{\varphi} \big)(\xi)&=&\frac{1}{h_n}\Big( S(X_{\varphi}(1_{[0,t+h)}))(\xi) - S(X_{\varphi}(1_{[0,t)}))(\xi)\Big)\\
                &=&\frac{ \langle \xi, \mathcal{M}_{-}^{(\varphi)}1_{[0,t+h_n)}\rangle-\langle \xi, \mathcal{M}_{-}^{(\varphi)}1_{[0,t)}  \rangle }{h_n}\\
                &\overset{*}{=}&\frac{ \langle \mathcal{M}_{+}^{(\varphi)}\xi, 1_{[0,t+h_n)}\rangle-\langle \mathcal{M}_{+}^{(\varphi)}\xi, 1_{[0,t)}  \rangle}{h_n}\\
                &=&\frac{\langle \mathcal{M}_{+}^{(\varphi)}\xi, 1_{[t,t+h_n)}\rangle}{h_n}=\frac{\int_t^{t+h_n} (\mathcal{M}_{+}^{(\varphi)}\xi)(x) dx}{h_n},
            \end{eqnarray*}
            using Lemma \ref{lem:IntegrPartsFormul} in $*$. We compute its limit, thanks to the continuity of $(\mathcal{M}_{+}^{(\varphi )}\xi) (\cdot)$ (see Remark \ref{rem:FractIntegralSmooth}),  as
            \[   \lim_{n \to \infty}\frac{1}{h_n} \int_t^{t+h_n}(\mathcal{M}_{+}^{(\varphi )}\xi) (x)dx=(\mathcal{M}_{+}^{(\varphi )}\xi) (t), \]
            so that $\{S \big(\mathscr{N}_{t,n}^{\varphi} \big)(\xi)\}_{n \in \mathbb{N}}$ is a Cauchy sequence.\\
            \item For each $n$, it is easy to prove that $S(\mathscr{N}_{t,n}^{\varphi} \big)(\cdot)=h_n^{-1}\langle \cdot, \mathcal{M}_{-}^{(\varphi )}1_{[t,t+h_n)} \rangle $
            is holomorphic on $U_{p,q}$. Now we prove the boundedness of the $S$-transform in a neighborhood of zero inside $U_{p,q}$.
            We apply the $S$-transform of $\mathscr{N}_{t,n}^{\varphi}$, for $\xi \in U_{p,q}$:  \[  S(\mathscr{N}_{t,n}^{\varphi})(\xi)=\frac{1}{h_n} \int_t^{t+h_n}(\mathcal{M}_{+}^{(\varphi )}\xi) (x)dx .   \]
Moreover, applying Lemma \ref{lem:ContinuityFracOpera}, we get that there exist $p' \in \mathbb{N}$ and $C_\varphi>0$ such that
            \[ h_n^{-1}\int_{t}^{t+h_n} |\mathcal{M}_{+}^{(\varphi)} \xi(x)| dx\leq \max_{x \in \mathbb{R}}(|\mathcal{M}_{+}^{(\varphi)} \xi(x)|)\leq C_\varphi\|\xi\|_{p'},\]
            for $\xi \in U_{p',q}\subset U_{p,q}$. Hence, we get that

            \[  |S \big(\mathscr{N}_{t,n}^{\varphi} \big)(\xi)|\leq C_\varphi\|\xi\|_{p'},\]
            for each $n \in \mathbb{N}$.
        \end{enumerate}
        Finally, we apply Theorem \ref{thm:CharctConvergSequenceHidaSpace} to establish the convergence of $\mathscr{N}_{t,n}^{\varphi}$ to an element in $(\mathcal{S})^{-1}$
       (denoted by $\mathscr{N}^{\varphi}_t$) fulfilling

        \[ S(\mathscr{N}^{\varphi}_t)(\xi)=\lim_{n \to \infty} S(\mathscr{N}_{t,n}^{\varphi})(\xi), \quad \xi \in U_{p',q}. \qedhere \]
    \end{proof}
    \begin{remark}
        A further analysis of the operators $\mathcal{M}^{(\varphi)}_{\pm}$ would be needed to prove the existence of noise by construction in Hida distribution space (see Appendix~\ref{App:HidaDistribution}).
    \end{remark}

    \begin{definition}
        Let the assumptions \textbf{(A1)-(A2)} hold for $\nu$ and  \textbf{(B1)-(B2)} hold for $\kappa$, then we define the BG-noise as the derivative of $X_\varphi(1_{[0,t)})$ in $(\mathcal{S})^{-1}$ and we denote it as $(\mathscr{N}^\varphi_t)_{t\geq0}$.
    \end{definition}

    \subsection{BG-Ornstein-Uhlenbeck Process}

    In this subsection we introduce the process through the following integral equation:

    \begin{equation}\label{eq:OUIntegralForm} U^{\varphi}_t=u_0-\theta \int_0^t U^{\varphi}_sds + \sigma X_{\varphi}(1_{[0,t)}), \quad t\geq 0, \end{equation}
    where $\theta >0$ and $\sigma \in \mathbb{R}$, following the methodology given in \cite{BOC1}.\\
    By applying the $S$-transform to \eqref{eq:OUIntegralForm} thanks to 
    Theorem \ref{thm:STransfoIsomo}, we find the differential equation solved by the $S$-transform of the process.
    For $\xi \in U_{p,q}=\{\xi \in \mathcal{S}_{\mathbb{C}} \mid 2^q \|\xi\|_p^2<1 \}$, we have that
    \begin{eqnarray*}
        S(U^{\varphi}_t)(\xi)&=&u_0-\theta S\left(\int_0^t U^{\varphi}_s ds \right)(\xi) + \sigma S( X_{\varphi}(1_{[0,t)}))(\xi)\\
        &\overset{*}{=}&u_0-\theta \int_0^t S(U^{\varphi}_s)(\xi) ds + \sigma  S( X_{\varphi}(1_{[0,t)}))(\xi)\\
        &\overset{**}{=}&u_0-\theta \int_0^t S(U^{\varphi}_s)(\xi) ds + \sigma \langle \xi, \mathcal{M}^{(\varphi)}_{-}1_{[0,t)}\rangle,\\
    \end{eqnarray*}
    where we use Theorem~6 in \cite{KON} in $*$ and Equation~\eqref{eq:STransformBrown} in $**$.\\

    We denote $u(t):=S(U^{\varphi}_t)(\xi)$, so that we have

    \[  u(t)=u_0-\theta \int_0^t u(s) ds + \sigma \int_0^t \left(\mathcal{M}^{(\varphi)}_{+}\xi \right) (x)dx,  \]
    for $t \geq 0$, which is equivalent to

    \[  \begin{cases}
        u'(t)=-\theta u(t)+\sigma \left(\mathcal{M}^{(\varphi)}_{+}\xi \right) (t), & t> 0,\\
        u(0)=u_0, & t=0,
    \end{cases} \]
   since $\left(\mathcal{M}^{(\varphi)}_{+}\xi \right) (t)$ is continuous, by Remark~\ref{rem:FractIntegralSmooth}.
    We solve the above ODE, obtaining

    \begin{eqnarray*}  u(t)&=&u_0e^{-\theta t} + \sigma \left(\int_0^t \left(\mathcal{M}^{(\varphi)}_{+}\xi \right) (s) ds - \theta \int_0^t e^{\theta (s-t)} \int_0^s \left(\mathcal{M}^{(\varphi)}_{+}\xi \right) (u) du ds  \right)\\
        &=&u_0e^{-\theta t} + \sigma \left(\langle \xi, \mathcal{M}^{(\varphi)}_{-}1_{[0,t)}\rangle - \theta \int_0^t e^{\theta (s-t)} \langle \xi, \mathcal{M}^{(\varphi)}_{-}1_{[0,s)}\rangle ds  \right)\\ 
        &\overset{*}{=}&u_0e^{-\theta t} + \sigma S(X_{\varphi}(1_{[0,t)}))(\xi) - \theta \int_0^t e^{\theta (s-t)} S(X_{\varphi}(1_{[0,s)}))(\xi) ds,\\ 
        \end{eqnarray*}
    where we use Equation~\eqref{eq:STransformBrown} in $*$. By applying Theorem 6 in \cite{KON} on the third term, we invert the $S$-transform in order to obtain

    \[ U^{\varphi}_t=u_0 e^{-\theta t} + \sigma X_{\varphi}(1_{[0,t)}) -  \sigma \theta \int_0^t e^{\theta (s-t)} X_{\varphi}(1_{[0,s)}) ds, \quad t\geq 0.  \]

    \begin{definition}
        Let $\theta>0$ and $\sigma \in \mathbb{R}$. The solution $U^{\varphi}_t$, $t\geq0$, of the integral equation~\eqref{eq:OUIntegralForm} is called BG-Ornstein-Uhlenbeck process. It is a Gaussian process with characteristic function, expectation and covariance function given, respectively, by
        \begin{eqnarray}
        &&\mathbb{E}\left(e^{\mathrm{i} k U^{\varphi}_t }\right)=\exp\left(\mathrm{i}ku_0e^{-\theta t} - \frac{k^2}{2}\|h_{\varphi,t}\|^2 \right), \notag\\
         && \mathbb{E}\left(U^{\varphi}_t\right)=u_0 e^{-\theta t},\notag\\
          && Cov(U^{\varphi}_{t_1} U^{\varphi}_{t_2})=\sigma^2\langle h_{\varphi,t_1} , h_{\varphi,t_2} \rangle, \quad t_1,t_2 \geq 0, \notag 
          \end{eqnarray} 
        where $h_{\varphi,t}(x)=\mathcal{M}^{(\varphi)}_{-}1_{[0,t)}(x) - \theta \int_0^t e^{\theta (s-t)} \mathcal{M}^{(\varphi)}_{-}1_{[0,s)}(x) ds$, with $x \in \R$.
    \end{definition}

\section{Asymptotic behavior of the variance and covariance
functions}\label{SecVar}

We now present a different representation of the
generalized process, in the
Mandelbrot-van Ness sense: let $t\geq 0$, $u\in \mathbb{R}$ and let%
\begin{equation*}
k_{\varphi }(t,u):=\left\{
\begin{array}{c}
\sqrt{C_{\nu }}[\nu ((t-u)_{+})-\nu ((-u)_{+})],\qquad \text{if }\mathcal{M}%
_{-}^{(\varphi )}=\mathcal{M}_{-}^{(\nu )} \\
\sqrt{C_{\kappa }}\left[ \chi ((t-u)_{+})-\chi ((-u)_{+}\right]
,\qquad
\text{if }\mathcal{M}_{-}^{(\varphi )}=\mathcal{M}_{-}^{(\kappa )}%
\end{array}%
\right.
\end{equation*}%
then, we define the following process
 \[B_{t}^{\varphi }:=\int_{\mathbb{R}%
}k_{\varphi }(t,u)dW_{u}, \qquad t\geq 0,
\]
where $W_{t}$, for $t\in \mathbb{R},$
is the two-sided Wiener process.

Therefore, the process $B_{t}^{\varphi }$ is Gaussian and such that $\mathbb{%
E}B_{t}^{\varphi }=0$, for any $t\geq 0$. Recall that the integral $\int_{\mathbb{R}%
}k_{\varphi }(t,u)^{2}du<+\infty $, for any $t \geq 0$, by Theorems
\ref{L2} and \ref{kap}, under $\textbf{(A1)}$-$\textbf{(A2)}$ and $\textbf{(B1)}$- $\textbf{(B2)}$, respectively. 
Thus we can write, when $\mathcal{M}%
_{-}^{(\varphi )}=\mathcal{M}_{-}^{(\nu )} $, that
\begin{eqnarray}
\mathbb{E}\left( B_{t}^{\varphi }\right) ^{2}
&=&\int_{\mathbb{R}}k_{\varphi }(t,u)^{2}du=C_{\nu }\left\{
\int_{\mathbb{-\infty }}^{0}\left[ \nu (t-u)-\nu (-u)\right]
^{2}du+\int_{0}^{t}\nu (t-u)^{2}du\right\}
\label{asy} \\
&=&C_{\nu }\left\{ \int_{0}^{+\infty }\left[ \int_{0}^{t}\overline{\nu }
(z+u)dz\right] ^{2}du+\int_{0}^{t}\nu (u)^{2}du\right\} .  \notag
\end{eqnarray}
Analogously, when $\mathcal{M}%
_{-}^{(\varphi )}=\mathcal{M}_{-}^{(\kappa )} $, the variance reads%
\begin{equation}
\mathbb{E}\left( B_{t}^{\varphi }\right) ^{2}=C_{\kappa }\left\{
\int_{0}^{+\infty }\left[ \int_{0}^{t}\kappa (z+u)dz\right]
^{2}du+\int_{0}^{t}\chi (u)^{2}du\right\}.  \label{variance}
\end{equation}
In (\ref{asy}) and (\ref{variance}), the constants $C_{\nu}$ and $C_{\kappa}$  are equal to $\int_{\mathbb{R}%
}k_{\varphi }(1,u)^{2}du<+\infty $, for $\mathcal{M}%
_{-}^{(\varphi )}=\mathcal{M}_{-}^{(\nu )}$ and $\mathcal{M}%
_{-}^{(\varphi )}=\mathcal{M}_{-}^{(\kappa )}$, respectively.
Moreover, we have that, for $t,h>0,$%
\begin{eqnarray*}
&&B_{t+h}^{\varphi }-B_{t}^{\varphi }=\int_{\mathbb{R}}\left[
k_{\varphi
}(t+h,u)-k_{\varphi }(t,u)\right] dW_{u} \\
&=&\int_{-\infty }^{t}\left[ k_{\varphi }(t+h,u)-k_{\varphi
}(t,u)\right] dW_{u}+\int_{t}^{t+h}k_{\varphi
}(t+h,u)dW_{u}=:[I_{1}+I_{2}].
\end{eqnarray*}%
For the first integral the following equality in distribution
holds, by
considering that $W_{t}$ has stationary increments:%
\begin{equation*}
I_{1}=\int_{-\infty }^{t}\left[ k_{\varphi }(t+h,u)-k_{\varphi
}(t,u)\right] dW_{u}\overset{d}{=}\int_{-\infty }^{0}\left[
k_{\varphi }(h,z)-k_{\varphi }(0,z)\right] dW_{z}
\end{equation*}%
and, analogously, $I_{2}\overset{d}{=}\int_{0}^{h}k_{\varphi
}(h,z)dW_{z}.$
Thus we can write that $B_{t+h}^{\varphi }-B_{t}^{\varphi }\overset{d}{=}%
B_{h}^{\varphi }$ and the covariance of $B_{t}^{\varphi }$ reads%
\begin{eqnarray}
\mathrm{Cov}\left[ B_{t}^{\varphi },B_{s}^{\varphi }\right]  &=&\mathbb{E}%
\left( B_{t}^{\varphi }B_{s}^{\varphi }\right) =\frac{1}{2}\left[ \mathbb{E}%
\left( B_{t}^{\varphi }\right) ^{2}+\mathbb{E}\left(
B_{s}^{\varphi }\right) ^{2}-\mathbb{E(}B_{t}^{\varphi
}-B_{s}^{\varphi })^{2}\right]   \label{cov}
\\
&=&\frac{1}{2}\left[ \mathbb{E}\left( B_{t}^{\varphi }\right) ^{2}+\mathbb{E}%
\left( B_{s}^{\varphi }\right) ^{2}-\mathbb{E(}B_{|t-s|}^{\varphi })^{2}%
\right] .  \notag
\end{eqnarray}%
Since both $B_{t}^{\varphi }$ and $X_{\varphi }(1_{[0,t)})$ are
centered Gaussian processes, in order to check that they are equal
in the sense of the finite dimensional distributions, we only need
to prove the equivalence of the covariance. In the first case,
i.e. when $\mathcal{M}_{-}^{(\varphi
)}=\mathcal{M}_{-}^{(\nu )},$ by recalling Theorem \ref{thmFourierconvolution} and applying the Plancherel's theorem:

\begin{eqnarray*}
&&\mathrm{Cov}\left[ B_{t}^{\varphi },B_{s}^{\varphi }\right]  \\
&=&\frac{1}{2}C_{\nu }\int_{\mathbb{R}}\left[ \nu (t-u)_{+}-\nu (-u)_{+}%
\right] ^{2}du+\frac{1}{2}C_{\nu }\int_{\mathbb{R}}\left[ \nu
(s-u)_{+}-\nu
(-u)_{+}\right] ^{2}du- \\
&&-\frac{1}{2}C_{\nu }\int_{\mathbb{R}}\left[ \nu (|t-s|-u)_{+}-\nu (-u)_{+}%
\right] ^{2}du \\
&=&\frac{1}{2}C_{\nu }\int_{\mathbb{R}}\varphi (ix)\widehat{1_{[0,t)}}(x)%
\overline{\varphi (ix)\widehat{1_{[0,t)}}(x)}dx+\frac{1}{2}C_{\nu }\int_{%
\mathbb{R}}\varphi (ix)\widehat{1_{[0,s)}}(x)\overline{\varphi (ix)\widehat{%
1_{[0,s)}}(x)}dx- \\
&&-\frac{1}{2}C_{\nu }\int_{\mathbb{R}}\varphi (ix)\widehat{1_{[0,|t-s|)}}(x)%
\overline{\varphi (ix)\widehat{1_{[0,|t-s|)}}(x)}dx \\
&=&\frac{1}{2}\left[ \mathbb{E}\left( X_{\varphi }(1_{[0,t)})\right) ^{2}+%
\mathbb{E}\left( X_{\varphi }(1_{[0,s)})\right)
^{2}-\mathbb{E(}X_{\varphi
}(1_{[0,|t-s|)}))^{2}\right]  \\
&=&\mathrm{Cov}\left( X_{\varphi }(1_{[0,t)}),X_{\varphi
}(1_{[0,s)})\right) .
\end{eqnarray*}%
In the last step, we have considered the stationarity of the increments of $%
X_{\varphi }(1_{[0,t)})$. The case where $\mathcal{M}_{-}^{(\varphi )}=%
\mathcal{M}_{-}^{(\kappa )}$ can be treated analogously, by recalling that $%
\int_{\mathbb{R}}e^{ixu}\left[ \chi (t-u)_{+}-\chi (-u)_{+}\right]
du=\varphi (ix)^{-1}\widehat{1_{[0,t)}}(x).$

We start by considering the asymptotic behavior of the variance function, in the first case, i.e. for $%
B_{t}^{\varphi }=B_{t}^{\nu}$, which was given in (\ref{asy}); it varies in accordance to the tail L%
\'{e}vy measure $\nu$ chosen in the definition of the process.

According to \eqref{asy}, we present $\mathbb{E}\left( B_{t}^{\nu }\right)
^{2}$ as the sum of two terms: $\mathbb{E}\left( B_{t}^{\nu }\right)
^{2}=C_{\nu }(I_{t}^{1}+I_{t}^{2}), $ where $I_{t}^{1}:=\int_{0}^{+\infty }%
\left[ \int_{0}^{t}\overline{\nu }(z+u)dz\right] ^{2}du$ and $%
I_{t}^{2}:=\int_{0}^{t}\nu (u)^{2}du$.

Let us start with the asymptotic behavior of $I_{t}^{2}$. Evidently, it is
an increasing function of $t$; therefore, in principle, it can tend to some
positive real number or to infinity. Let's prove a simple result
demonstrating that both possibilities can be realized.

\begin{lem}
\label{lem1} Let $f(t)$, $t \in \mathbb{R}^{+}\backslash\{0\}$ be a function
in $\mathcal{C}^{1}(\mathbb{R}^{+}\backslash \{0\})$ such that

(\textbf{\emph{$\lozenge$}}) $f(t)$ is increasing, convex and its second
derivative $f^{\prime\prime}(t)$ exists a.e. w.r.t. the Lebesgue measure and
belongs to $L^1(\varepsilon, \infty)$ for any $\varepsilon>0 $. Moreover, $%
f(t)=o(t)$, as $t\rightarrow +\infty $.

Then the following holds true

\begin{itemize}
\item[$(i)$] The function $I_{t}^{2}$ satisfies the conditions in (\textbf{%
\emph{$\lozenge$}})

\item[$(ii)$] Any function $f(t)$ in $\mathcal{C}^{1}(\mathbb{R}^{+}\backslash \{0\})$
satisfying (\textbf{\emph{$\lozenge$}}) can be considered as $I_{t}^{2}$.

\item[$(iii)$] The following asymptotic behaviors are both possible, as $%
t\rightarrow +\infty :$ $I_{t}^{2}\rightarrow C\in \mathbb{R}^{+}$ and $%
I_{t}^{2}\rightarrow +\infty $.
\end{itemize}
\end{lem}

\begin{proof}
\begin{itemize}
\item[$(i)$] The facts that the function $I_{t}^{2}$ is an
increasing convex function in $\mathcal{C}^{1}(\mathbb{R}^{+}\backslash
\{0\})$ such that its first derivative decreases and the second
derivative $(I^{2}_{t})^{\prime\prime}$ exists a.e. w.r.t. the
Lebesgue measure obviously follow by considering that
$(I^{2}_{t})^{\prime}=\nu^2(t)$, for any $t>0$, and $(I^{2}_{t})^{\prime%
\prime}=2\nu(t)\overline{\nu}(t)$, for a.e. $t>0$ w.r.t. the Lebesgue
measure. Furthermore, for any $\varepsilon>0$ and all $t\ge \varepsilon$, we
have that $\nu(t)\overline{\nu}(t)\le \nu(\varepsilon)\overline{\nu}(t)$ and
then $(I^{2}_{t})^{\prime\prime}$ belongs to $L^1(\varepsilon, \infty)$ for
any $\varepsilon>0 $. The relation $I_{t}^{2}=\int_{0}^{t}\nu
(u)^{2}du=o(t), $ $t\rightarrow +\infty ,$ easily follows by applying the
l'H\^{o}pital rule, since $\lim_{t\rightarrow +\infty
}I_{t}^{2}/t=\lim_{t\rightarrow +\infty }\nu (t)^{2}=0.$

\item[$(ii)$] Let $f\in \mathcal{C}^{1}(\mathbb{R}^{+}\backslash \{0\})$ satisfy (%
\textbf{\emph{$\lozenge$}}) and put $\nu (t)=(f^{\prime }(t))^{1/2}$. Then $%
\nu (t)\rightarrow 0$ as $t\rightarrow \infty$, according to l'H\^{o}pital rule.
Assume, without loss of generality, that $\nu (t)>0$, for all $t>0$. Trying
to represent $\nu (t)=\int_{t}^{+\infty }\overline{\nu }(u)du$, we get that
\begin{equation*}
\overline{\nu }(t)=-\frac{1}{2}\frac{f^{\prime \prime }(t)}{(f^{\prime
}(t))^{1/2}}=-\frac{1}{2}\frac{f^{\prime \prime }(t)}{\nu (t)}
\end{equation*}
a.e. w.r.t. the Lebesgue measure and $\overline{\nu }(t)\ge 0$. Obviously, $%
\overline{\nu }\in L^{1}(\epsilon ,+\infty ) $ for any $\varepsilon>0$.

\item[$(iii)$] The case where $I_{t}^{2}\rightarrow C\in \mathbb{R}^{+}$, $%
t\rightarrow +\infty ,$ is obtained, for example, when $\nu (u)=e^{-u}$,
while, for $\nu (u)=u^{(\alpha -1)/2},$ $\alpha \in (0,1),$ we have that $%
I_{t}^{2}\rightarrow +\infty $, $t\rightarrow +\infty $ (see Section \ref{SubSec: SpecialCases}, for details on these special cases).
\end{itemize}
\end{proof}

We prove now that $I_{t}^{1}$ does not increase faster than $I_{t}^{2}$, as $%
t\rightarrow \infty.$

\begin{lem}
\label{lem2} For the integrals $I_{t}^{1}$ and $I_{t}^{2}$ the following
relationships hold:

\begin{enumerate}
\item[$(i)$] For any $t>0$, $I_{t}^{1}\leq 2I_{t}^{2}.$

\item[$(ii)$] If $I_{t}^{2}\rightarrow C>0$, as $t\rightarrow +\infty ,$
then $I_{t}^{1}\rightarrow C$.

\item[$(iii)$] For $t\rightarrow +\infty$, $I_{t}^{2}\rightarrow +\infty $
if and only if $I_{t}^{1}\rightarrow +\infty .$
\end{enumerate}
\end{lem}

\begin{proof}
\begin{enumerate}
\item[$(i)$] Recall that for any $t,u>0$ $\nu (u)>\nu (t+u)$ and for any $%
t>0 $ rewrite and bound $I_{t}^{1}$ as follows:
\begin{eqnarray*}
I_{t}^{1} &=&\int_{0}^{t}[\nu (u)-\nu (t+u)]^{2}du+\int_{t}^{+\infty }[\nu
(u)-\nu (t+u)]^{2}du \\
&\leq &\int_{0}^{t}\nu (u)^{2}du+\int_{t}^{+\infty }\left( \int_{u}^{t+u}%
\overline{\nu }(s)ds\int_{u}^{t+u}\overline{\nu }(r)dr\right) du \\
&=&:I_{t}^{2}+I_{t}^{3}.
\end{eqnarray*}
Now let us provide an upper bound for $I_{t}^{3}$. According to the Fubini
theorem, we can write
\begin{eqnarray*}
I_{t}^{3} &=&\int_{t}^{+\infty }\overline{\nu }(s)ds\int_{t}^{+\infty }%
\overline{\nu }(r)drdu\int_{u\in (r-t,r)\cap (s-t,s)}\leq t\left(
\int_{t}^{+\infty }\overline{\nu }(s)ds\right) ^{2} \\
&=&t\nu (t)^{2}\leq \int_{0}^{t}\nu (u)^{2}du=I_{t}^{2}.
\end{eqnarray*}

\item[$(ii)$] If $I_{t}^{2}\rightarrow C $ as $t\rightarrow \infty$, it is
sufficient to mention again that $\nu (t+u)$ decreases (to zero) in $t$, and
$\nu (u)\geq \nu (u)-\nu (t+u)\geq 0$, therefore, by the Lebesgue monotone
convergence theorem $\lim_{t\rightarrow +\infty }I_{t}^{1}=\int_{0}^{+\infty
}\nu (u)^{2}du=C.$

\item[$(iii)$] If $I_{t}^{2}\rightarrow +\infty,$ then by the Lebesgue
monotone convergence theorem $\lim_{t\rightarrow +\infty
}I_{t}^{1}=\int_{0}^{+\infty }\nu (u)^{2}du=+\infty .$ The converse statement follows
from (ii).
\end{enumerate}
\end{proof}

As an immediate consequence of the previous lemmas, we have the
following result about  the variance's asymptotics.

\begin{coro}\label{corovar}
For the variance of $B_{t}^{\nu}$ we have that

\begin{enumerate}
\item[$(i)$] if $\nu \in L^2 (\mathbb{R}^+ )$, then $\lim_{t \rightarrow
+\infty}\mathrm{Var}\left( B_{t}^{\nu }\right)=C>0$

\item[$(ii)$] if $\nu \notin L^2 (\mathbb{R}^+ )$, then $\lim_{t \rightarrow
+\infty}\mathrm{Var}\left( B_{t}^{\nu }\right)=+\infty$ and $\lim_{t
\rightarrow +\infty}\frac{\mathrm{Var}\left( B_{t}^{\nu }\right)}{t}=0.$
\end{enumerate}
\end{coro}

\begin{remark}
Consider now additionally the natural question whether the two integrals $%
I_{t}^{1}$ and $I_{t}^{2}$ have always the same order at infinity or whether
instead it is possible that $I_{t}^{1}=o(I_{t}^{2})$, as $t \rightarrow
+\infty$.We will show with the help of examples that both possibilities can be realized.

\begin{example}
Let $\nu (x)=\frac{(2/x)^{1/4}}{\sqrt{\log 2}}%
1_{\{0<x<2\}+\frac{1}{\sqrt{\log x}}1_{\{x\geq 2\}}}$; it is easy to check that it satisfies the assumptions \textbf{(A1)-(A2)}. Then, for $t>2$, we write%
\begin{eqnarray*}
I_{t}^{1} &=&\int_{0}^{2}\left(\frac{(2/x)^{1/4}}{\sqrt{\log 2}}-\frac{1}{\sqrt{\log
(t+x)}}\right) ^{2}dx+\int_{2}^{+\infty }\left( \frac{1}{\sqrt{\log x}}-%
\frac{1}{\sqrt{\log (t+x)}}\right) ^{2}dx \\
&=&:J_{t}^{\prime }+J_{t}^{\prime \prime }.
\end{eqnarray*}%
Moreover, $J_{t}^{\prime }\leq \int_{0}^{2}\frac{\sqrt{2}}{\sqrt{x}\log 2}dx=\frac{4}{\log 2%
}$ and
\begin{eqnarray*}
J_{t}^{\prime \prime } &=&\int_{2}^{+\infty }\frac{(\sqrt{\log (t+x)}-\sqrt{%
\log x})^{2}}{\log x\log (t+x)}dx \\
&=&\int_{2}^{+\infty }\frac{\log ^{2}(1+t/x)dx}{\log x\log (t+x)(\sqrt{\log
(t+x)}+\sqrt{\log x})^{2}} \\
&\leq &\frac{1}{\log 2}\int_{2}^{+\infty }\frac{\log ^{2}(1+t/x)dx}{\log
^{2}(t+x)} \\
&=&\frac{t}{\log 2}\int_{2/t}^{+\infty }\frac{\log ^{2}(1+1/z)dz}{(\log
t+\log (1+z))^{2}} \\
&\leq &\frac{t}{\log 2\log ^{2}t}\int_{0}^{+\infty }\log ^{2}(1+1/z)dz,
\end{eqnarray*}%
and the latter integral is finite, as can be checked by the following steps:%
\begin{eqnarray*}
\int_{0}^{+\infty }\log ^{2}(1+1/z)dz &=&\int_{0}^{1}\log
^{2}(1+1/z)dz+\int_{1}^{+\infty }\log ^{2}(1+1/z)dz \\
&\leq &\int_{1}^{+\infty }\frac{1}{r^{2}}\log ^{2}(1+r)dr+\int_{1}^{+\infty }%
\frac{1}{z^{2}}dz<\infty .
\end{eqnarray*}%
Then, we have that $I_{t}^{1}\leq Ct/\log ^{2}t$, for $C>0,$ while $I_{t}^{2}\geq (t-2)/\log t$, for $t>2$, so that $\lim_{t\rightarrow
+\infty }I_{t}^{1}/I_{t}^{2}\leq \lim_{t\rightarrow +\infty }1/\log t=0.$
\end{example}

\begin{example} \label{Exx}
Consider now the case of the fractional Brownian motion with $%
H=\alpha /2<1/2 $, i.e. where $\overline{\nu }%
(t)=\frac{t^{(\alpha -3)/2}}{\Gamma ((\alpha -1)/2)},$ with $\alpha \in (0,1)
$, and ${\nu }(t)=\frac{t^{(\alpha -1)/2}}{\Gamma ((\alpha +1)/2)}$. Then, we have that $%
I_{t}^{2}=C'_{\alpha}{t^{\alpha }}$, for $C'_{\alpha}>0$. On the other hand, we
have that 
\begin{eqnarray}
I_{t}^{1}&=&C''_{\alpha}t^{\alpha }\int_{0}^{\infty }\left[
x^{(\alpha -1)/2}-(x+1)^{(\alpha -1)/2}\right] ^{2}dx \notag \\
&=&C''_{\alpha}t^{\alpha }\left[\frac{\Gamma ^{2}((\alpha +1)/2)}{\sin (\pi \alpha /2)\Gamma (\alpha
+1)}-\frac{1}{\alpha }\right], \notag
\end{eqnarray}
(for $C'_{\alpha}>0$), according to \cite{MIS}, Theorem 1.3.1. So, in this example, $I_{t}^{i},i=1,2
$ have the same order of growth to infinity.

\end{example}
\end{remark}

We now consider the asymptotic behavior of the covariance. For technical
simplicity, we restrict ourselves to the following form of covariance $%
\rho_n :=\mathbb{E}B_{1}^{\nu}(B_{n+1}^{\nu}-B_{n}^{\nu})$, for $n \in
\mathbb{N}$. We concentrate on the following two questions:

\begin{enumerate}
\item[$(i)$] Is $\rho_n$ positive (persistence property) or negative
(anti-persistence property)?

\item[$(ii)$] Does the series $\sum_{n=1}^{+\infty}|\rho_n|$ converge (short
memory property) or diverge (long-memory property)?
\end{enumerate}

To answer these questions, we evaluate $\rho_n$, according to formula (\ref%
{cov}) and assuming that $\int_{0^+}^{+\infty}\nu(x)dx<\infty$,

\begin{eqnarray}
\rho_n &=&\mathbb{E}(B_{1}^{\nu}B_{n+1}^{\nu})-\mathbb{E}(B_{1}^{\nu}B_{n}^{%
\nu}) \label{ro} \\ 
&=&\frac{1}{2}\left[ \mathbb{E}\left( B_{1}^{\nu }\right) ^{2}+\mathbb{E}%
\left( B_{n+1}^{\nu }\right) ^{2}-\mathbb{E(}B_{n}^{\nu })^{2}- \mathbb{E}%
\left( B_{1}^{\nu }\right) ^{2}-\mathbb{E}\left( B_{n}^{\nu }\right) ^{2}+%
\mathbb{E(}B_{n-1}^{\nu })^{2}\right] \notag \\ 
&=&\frac{1}{2}\left[ \mathbb{E}\left( B_{n+1}^{\nu }\right) ^{2}+\mathbb{E}%
\left( B_{n-1}^{\nu }\right) ^{2}-2\mathbb{E(}B_{n}^{\nu })^{2}\right]. \notag
\end{eqnarray}%

By considering (\ref{asy}), we have that
\begin{eqnarray}
\rho _{n} &=&\frac{C_{\nu }}{2}\int_{0}^{\infty }\left( \left[ \nu (x)-\nu
(n+1+x)\right] ^{2}+\left[ \nu (x)-\nu (n-1+x)\right] ^{2}-2\left[ \nu
(x)-\nu (n+x)\right] ^{2}\right) dx  \notag \\
&+&\frac{C_{\nu }}{2}\left( \int_{0}^{n+1}\nu ^{2}(x)dx+\int_{0}^{n-1}\nu
^{2}(x)dx-2\int_{0}^{n}\nu ^{2}(x)dx\right)   \label{sum} \\
&=&\frac{C_{\nu }}{2}J_{n}+\frac{C_{\nu }}{2}\left( I_{n+1}^{+\infty
}+I_{n-1}^{+\infty }-2I_{n}^{+\infty
}+I_{0}^{n+1}+I_{0}^{n-1}-2I_{0}^{n}\right) ,  \notag
\end{eqnarray}%
where $J_{n}:=\int_{0}^{\infty }\nu (x)\left[ -\nu (n+1+x)-\nu (n-1+x)+2\nu
(n+x)\right] dx$ and $I_{\alpha (n)}^{\beta (n)}:=\int_{\alpha (n)}^{\beta
(n)}\nu ^{2}(x)dx$. The last sum of integrals in (\ref{sum}) vanishes if we
treat the improper integrals as $I_{\alpha (n)}^{+\infty
}=\lim_{N\rightarrow +\infty }I_{\alpha (n)}^{N}$, while $J_{n}$ converges
under our assumption $\mathbf{(A2)}$ on $\nu $.

Recall now that $\nu(\beta(n))-\nu(\alpha(n))= \int_{\alpha(n)}^{\beta(n)}
\overline{\nu}(x)dx$, then, by assuming that $\overline{\nu}(\cdot)$ is
decreasing, we have that $\rho_n =C_{\nu }J_n /2 <0.$ Therefore, the
``natural" property of the process $B_{t}^{\nu}$ is the anti-persistence.

However, in general, this property can be violated for some $n$: since $%
\int_{\epsilon}^{+\infty}\overline{\nu}(s)ds< \infty$, it is impossible that
$\int_{n+x}^{n+x+1} \overline{\nu}(s)ds$ is strictly increasing for $n>N_0$
and some $N_0$. Therefore the process can not be persistent for all $n>N_0$
and some $N_0$; however the case $\rho_n =0$ is possible for $n>N_0$.

We now answer to the second question, concerning the memory properties of
the process $B_{t}^{\nu}$.

\begin{theorem}
The process $B_{t}^{\nu }$ displays a short-range dependence, for any choice
of $\nu (\cdot )$ satisfying the assumptions $\textbf{ (A1)-(A2)}$.
\end{theorem}

\begin{proof}
According to the well-known definition of short-memory property, it is
necessary to prove that $\sum_{n=1}^{+\infty}|\rho_n|<\infty$; then we
represent $\rho_n$ in the following way

\begin{eqnarray*}
\rho _{n} &=&\frac{C_{\nu }}{2}\int_{0}^{\infty }\left[ \nu (x)-\nu (x+1)%
\right] \left[ \nu (n+x)-\nu (n+1+x)\right] dx+ \\
&&+\frac{C_{\nu }}{2}\int_{0}^{1}\nu (x)\left[ \nu (n+x)-\nu (n-1+x)\right]
dx,
\end{eqnarray*}%
where the first integral is non-negative and the second one is non-positive.
Therefore

\begin{eqnarray*}
|\rho _{n}| &\leq &\frac{C_{\nu }}{2}\int_{0}^{\infty }\left[ \nu (x)-\nu
(x+1)\right] \left[ \nu (n+x)-\nu (n+1+x)\right] dx+\\
&+&\frac{C_{\nu }}{2}%
\int_{0}^{1}\nu (x)\left[ \nu (n-1+x)-\nu (n+x)\right] dx \\
&=&:K_{n}^{1}+K_{n}^{2}.
\end{eqnarray*}%
Now it is sufficient to prove that $\sum_{n=1}^{+\infty }K_{n}^{i}<\infty $,
for $i=1,2$. Since, for any $N>1$ and according to the Fubini theorem, we
can write%
\begin{eqnarray*}
\sum_{n=N}^{+\infty }K_{n}^{1} &=&\frac{C_{\nu }}{2}\int_{0}^{\infty }\left[
\nu (x)-\nu (x+1)\right] \nu (N+x)dx \leq \frac{C_{\nu }}{2}\int_{0}^{\infty
}\left[ \nu (x)-\nu (x+1)\right]\nu (1+x)dx \\
&=&\frac{C_{\nu }}{2}\int_{0}^{\infty }\left( \int_{x}^{x+1}\overline{\nu }%
(s)ds\right) \nu (x+1)dx \\
&=&\frac{C_{\nu }}{2}\int_{0}^{1}\overline{\nu }(s)\left( \int_{0}^{s}\nu
(x+1)dx\right) ds+\int_{1}^{+\infty }\overline{\nu }(s)\left(
\int_{s-1}^{s}\nu (x+1)dx\right) ds \\
&\leq &\frac{\nu (1)C_{\nu }}{2}\int_{0}^{1}s\overline{\nu }(s)ds+\nu
^{2}(1)<\infty ,
\end{eqnarray*}%
the series $\sum_{n=1}^{+\infty }K_{n}^{1}$ converges. Similarly, $%
\sum_{n=1}^{+\infty }K_{n}^{2}$ converges since, for any $N>1,$%
\begin{equation*}
\sum_{n=N}^{+\infty }K_{n}^{2}=\frac{C_{\nu }}{2}\int_{0}^{1}\nu (x)\nu
(N-1+x)dx\leq \int_{0}^{1}\nu ^{2}(x)dx<\infty .
\end{equation*}
\end{proof}

\begin{remark}
Let us consider the fractional Brownian motion case (for $H<1/2$) given in Example \ref{Exx}: then
\begin{eqnarray*}
|\rho _{n}| &\leq &\left\vert \int_{0}^{1}\left[ x^{(\alpha
-1)/2}-(x+1)^{(\alpha -1)/2}\right] \left[ (n+x)^{(\alpha
-1)/2}-(n+x+1)^{(\alpha -1)/2}\right] dx\right\vert + \\
&&+\left\vert \int_{1}^{+\infty }\left[ x^{(\alpha -1)/2}-(x+1)^{(\alpha
-1)/2}\right] \left[ (n+x)^{(\alpha -1)/2}-(n+x+1)^{(\alpha -1)/2}\right]
dx\right\vert \\
&+&\left\vert \int_{0}^{1}x^{(\alpha -1)/2}\left[ (n+x)^{(\alpha
-1)/2}-(n-1+x)^{(\alpha -1)/2}\right] dx\right\vert
=:L_{n}^{1}+L_{n}^{2}+L_{n}^{3}.
\end{eqnarray*}%
The integral $L_{n}^{1}$ is bounded since%
\begin{eqnarray*}
L_{n}^{1} &\leq &\int_{0}^{1}x^{(1-\alpha )/2}\frac{(n+1+x)^{(1-\alpha
)/2}-(n+x)^{(1-\alpha )/2}}{(n+x)^{(1-\alpha )/2}(n+x+1)^{(1-\alpha )/2}}dx
\\
&\leq &\int_{0}^{1}x^{(1-\alpha )/2}\frac{(1-\alpha )(n+x)^{-(\alpha +1)/2}}{%
2(n+x)^{(1-\alpha )/2}(n+x+1)^{(1-\alpha )/2}}dx\leq (1-\alpha )n^{(\alpha
-3)/2}\int_{0}^{1}x^{(1-\alpha )/2}dx
\end{eqnarray*}%
and, analogously, for $L_{n}^{3}.$ On the other hand, we have that%
\begin{eqnarray*}
L_{n}^{2} &=&\int_{1}^{+\infty }\frac{(x+1)^{(1-\alpha )/2}-x^{(1-\alpha )/2}%
}{x^{(1-\alpha )/2}(x+1)^{(1-\alpha )/2}}\frac{(n+x+1)^{(1-\alpha
)/2}-(n+x)^{(1-\alpha )/2}}{(n+x)^{(1-\alpha )/2}(n+x+1)^{(1-\alpha )/2}}dx
\\
&\leq &\frac{(1-\alpha )^{2}}{4}\int_{1}^{+\infty }\frac{x^{-(\alpha +1)/2}}{%
x^{(1-\alpha )/2}(x+1)^{(1-\alpha )/2}}\frac{(n+x)^{-(\alpha +1)/2}}{%
(n+x)^{(1-\alpha )/2}(n+x+1)^{(1-\alpha )/2}}dx \\
&\leq &(1-\alpha )^{2}n^{(\alpha -3)/2}\int_{1}^{+\infty }\frac{dx}{%
x(x+1)^{(1-\alpha )/2}}\sim n^{(\alpha -3)/2},
\end{eqnarray*}%
as $n\rightarrow +\infty .$
\end{remark}

We now consider the second case, i.e. the process $B_{t}^{\varphi }=B_{t}^{\kappa }$ and we rewrite its variance (given in (\ref{variance})) as
\begin{equation}
\mathbb{E}\left( B_{t}^{\kappa }\right) ^{2}=C_{\kappa }\left\{
\int_{0}^{+\infty }\left[\int_{u}^{t+u}\kappa (z)dz\right]
^{2}du+\int_{0}^{t}\left[\int_{0}^{u}\kappa (z)dz\right]^{2}du\right\}=:C_{\kappa }(J_t ^1 +J_t ^2 ). \label{va}
\end{equation}
Obviously, both $J_t ^1$ and $J_t ^2$ increase and tend to $+\infty$, for $t \to + \infty$, but, in the next examples, we show that the asymptotic behavior of the two terms can be equal, for large $t$.

\begin{example}
In the fractional Brownian motion case, with $H>1/2$, i.e. for $\kappa(z)=z^{\beta-3/2}$ with $\beta\in(1/2, 1)$, then $J_t ^1 / J_t ^2 \to C>0$, as $t \to + \infty$, since $J_t ^1=C_1 t^{2\beta}$ and $J_t ^2=C_2 t^{2\beta}$.
\end{example}

\begin{example}\label{Ex54}
Let $\kappa (z)=z^{-\delta}1_{0<z<1}+\frac{1}{z}%
1_{z \geq 1},$ for $\delta \in (0,1)$, then it can be checked that $\kappa(\cdot)$ and $\chi(\cdot)$ satisfy $\mathbf{(B1)}$ and $\mathbf{(B2)}$. The corresponding Bernstein function is given by $\varphi(\theta)=1/\tilde{\kappa}(\theta)$, with $\tilde{\kappa}(\theta)=\theta ^{\delta -1}\gamma(1-\delta; \theta)+E_1 (\theta)$, where $E_1(x)$ is the exponential function and $\gamma (\beta ;x):=\int_{0}^{x}e^{-w}w^{\beta -1}dw$ is the lower-incomplete gamma function. 
In this case, $J_t ^1 / J_t ^2 \to 0$, as $t \to + \infty$.  Indeed, for $t>1$,
\begin{eqnarray*}
J_{t}^{1} &=&\int_{0}^{1} \left[ \int_{u}^{1}z^{-\delta}dz+ \int_{1} ^{t+u} \frac{1}{z} dz \right]^2 du+  \int_{1}^{+\infty} \left[ \log(t+u)- \log u\right]^2 du \\
&=& \int_{0}^{1} \left[ \frac{(1-u)^{1-\delta}}{1-\delta} + \log (t+u) \right]^2 du+ \int_{1}^{+\infty} \log^2 \left(1+\frac{t}{u} \right) du \\
&\leq& \frac{2}{{(1-\delta)^2}}\int_{0}^{1} (1-u)^{2-2\delta}  du + 2 \log^2 \left(1+t \right) +t \int_{1/t}^{+\infty} \log^2 \left(1+\frac{1}{z} \right) dz \\
&=& \frac{2}{(3-2\delta)(1-\delta)^2} + 2 \log^2 \left(1+t \right) +t \int_{0}^{t} \frac{\log ^2 (1+v)}{v^2} dv \\
&\leq& C \left[1+ \log^{2} (t+1) +t\right],
\end{eqnarray*}%
since $\int_{0}^{t} \frac{\log ^2 (1+v)}{v^2} dv< \int_{0}^{+\infty} \frac{\log ^2 (1+v)}{v^2} dv =\pi^2 /3$.
On the other hand, we have that
\begin{equation*}
    J_{t}^{2} \geq \int_{1}^{t} \left[ \int_{1}^{s} \frac{1}{z} dz\right]^2 ds =\int_{1}^{t} \log^2 (s)ds
\end{equation*} 
and that
\begin{equation*}
\frac{1+\log^{2} (t+1) +t}{\int_{1}^{t} \log^2 (s)ds}\to 0,
\end{equation*}
as $t \to +\infty$, by the l'H\^opital rule.
\end{example}

\begin{example}\label{Ex55}
 Let $\kappa (z)=\frac{1}{\sqrt{z}}%
1_{0<z <e}+\frac{1}{\sqrt{z}\log z}1_{z\geq e},$ then $J_t ^1 / J_t ^2 \to +\infty$, as $t \to + \infty$. In  the following calculations, we can consider the integrals as starting from $e$ instead of $0$, without loss of generality. Let us again apply the L'H\^opital rule, twice:
\begin{eqnarray*}
   \lim_{t \to +\infty}\frac{J_t ^1}{J_t ^2} &=& \lim_{t \to +\infty}\frac{\int_{e}^{+\infty}
    \left(\int_{v}^{t+v}\kappa (s)ds \right)^2 dv}{\int_{e}^{t}
    \left(\int_{e}^{v}\kappa (s)ds \right)^2 dv} =  2\lim_{t \to +\infty}\frac{\int_{e}^{+\infty}
    \kappa(t+v)\int_{v}^{t+v}\kappa (s)ds  dv}{
    \left(\int_{e}^{t}\kappa (s) ds \right)^2 } \\
&=&  2\lim_{t \to +\infty}\frac{\int_{e}^{+\infty}
   \left[ \kappa(t+v)-\kappa(v)\right]\int_{v}^{t+v}\kappa (s)ds  dv+\int_{e}^{+\infty}
   \kappa(v)\int_{t}^{t+v}\kappa (s)ds  dv}{
    \left(\int_{e}^{t}\kappa (s)ds \right)^2 }\\
&=&-\lim_{t \to +\infty}\frac{\left(\int_{e}^{e+t}\kappa (s)ds \right)^2 }
    {\left(\int_{e}^{t}\kappa (s)ds \right)^2 }+\lim_{t \to +\infty}   \frac{\int_{e}^{+\infty}\kappa (t+v)\kappa(v)dv }{
    \kappa(t)\int_{e}^{t}\kappa (s) ds }.
\end{eqnarray*}
Since $\int_{t}^{t+e}\kappa (s)ds \leq e \kappa(t) \to 0$, as $t \to +\infty$, the first limit is finite. For the second one, by considering that $\kappa(\cdot)$ is non-increasing, we can write

\begin{equation*}
 \lim_{t \to +\infty}
   \frac{\int_{e}^{+\infty}\kappa (t+v)\kappa(v)dv }{
    \kappa(t)\int_{e}^{t}\kappa (s) ds } 
    \geq \lim_{t \to +\infty}
   \frac{\int_{t}^{+\infty}\kappa^2 (t+v) dv }{
    \kappa(t)\int_{e}^{t}\kappa (s) ds } =\lim_{t \to +\infty}
   \frac{\int_{2t}^{+\infty}\kappa^2 (v) dv }{
    \kappa(t)\int_{e}^{t}\kappa (s) ds }.
\end{equation*}
Let now substitute $\kappa (z)=1/(\sqrt{z}\log z) \in L^2 [e,+\infty) \backslash L^1 [e,+\infty)$, so that we have

\begin{equation*}
    \lim_{t \to +\infty}
   \frac{\int_{2t}^{+\infty}\kappa^2 (v) dv }{
    \kappa(t)\int_{e}^{t}\kappa (s) ds }=\lim_{t \to +\infty}
   \frac{\log^{-1}(2t) \log (t) \sqrt{t} }{
    \int_{e}^{t}ds/(\sqrt{s} log (s))} =+\infty,
\end{equation*}
since  $\log^{-1}(2t) \log (t) \to 1$ and by applying once again the l'H\^opital rule.

\end{example}

We now have the following result on the persistence property.
\begin{theorem}
The process $B_{t}^{\kappa }$ is persistent (i.e. $\rho_n >0,$ for $n \geq 1$), for any choice of $\kappa (\cdot)$ satisfying \textbf{(B1)-(B2)}. 
\end{theorem}
\begin{proof}
We prove that $\mathbb{E}\left( B_{t}^{\kappa }\right) ^{2}$ is a convex function; indeed, by differentiating (\ref{va}), we have that
\begin{eqnarray}
&&(J_t ^1 +J_t ^2 )'  \label{per} \\
&=&2
\int_{0}^{+\infty }\kappa(t+u)\int_{u}^{t+u}\kappa (z)dz  du+\left[\int_{0}^{t}\kappa (z)dz\right]^{2} \notag \\
&=& 2\int_{0}^{+\infty }\left[\kappa(t+u)-\kappa(u)\right]\int_{u}^{t+u}\kappa (z)dz du+\left[\int_{0}^{t}\kappa (z)dz\right]^{2}+ 2\int_{0}^{+\infty }\kappa(u)\int_{u}^{t+u}\kappa (z)dz du  \notag \\
&=& 2\int_{0}^{+\infty }\kappa(u)\int_{u}^{t+u}\kappa (z)dz du.  \notag
\end{eqnarray}
In the last step we have considered that
\begin{eqnarray}
  \int_{0}^{+\infty }\left[\kappa(t+u)-\kappa(u)\right]\int_{u}^{t+u}\kappa (s)ds du 
  &=&\int_{0}^{+\infty }\int_{u}^{t+u}\kappa (s)ds \frac{d}{du}\left[\int_{u}^{t+u}\kappa (s)ds\right]du \label{eqa} \\
   &=&\frac{1}{2} \left.  \left[\int_{u}^{t+u}\kappa (s)ds\right]^2 \right\vert_{u=0} ^{+\infty}=-\frac{1}{2}\left[\int_{0}^{t}\kappa (s)ds\right]^2. \notag
\end{eqnarray}
and that the integral $\int_{u}^{t+u}\kappa (z)dz$ decreases w.r.t. $u$ (since $\kappa(\cdot)$ is non-increasing). Moreover, $\lim_{u \to +\infty}\int_{u}^{t+u}\kappa (z)dz \leq t \lim_{u \to +\infty}\kappa(u)=0 $ (by (\ref{kappa2})), so that
\begin{equation*}
    (J_t ^1 +J_t ^2 )'  =2
\int_{0}^{+\infty }\kappa(u)\int_{u}^{t+u}\kappa (z)dz du,
\end{equation*}
which is increasing in $t$. The persistence of the process easily follows from (\ref{ro}).
\end{proof}
Let us now establish the long-range dependence. 
\begin{theorem}
The process $B_{t}^{\kappa }$ displays long-range dependence, for any choice
of $\kappa (\cdot )$ satisfying \textbf{(B1)-(B2)}.
\end{theorem}
\begin{proof}
Since, for any $n$, $\rho_n >0$, we can write, for $N \geq 2,$ that
\begin{eqnarray}
\sum_{n=2}^{N}\rho_n &=&\frac{1}{2}\sum_{n=2}^{N}\left[ \mathbb{E}\left( B_{n+1}^{\kappa }\right) ^{2}+\mathbb{E}%
\left( B_{n-1}^{\kappa }\right) ^{2}-2\mathbb{E}(B_{n}^{\kappa })^{2}\right] \notag \\
&=&\frac{1}{2}\left[ \mathbb{E}\left( B_{N+1}^{\kappa }\right) ^{2}-\mathbb{E}%
\left( B_{N}^{\kappa }\right) ^{2}+\mathbb{E}%
\left( B_{1}^{\kappa }\right) ^{2}-\mathbb{E(}B_{2}^{\kappa })^{2}\right].
\end{eqnarray}
For the first two terms, we have that
\begin{eqnarray*}
&&\mathbb{E}\left( B_{N+1}^{\kappa }\right) ^{2}-\mathbb{E}%
\left( B_{N}^{\kappa }\right) ^{2} \\
&=& C_{\kappa}  \int_{0}^{+\infty} \left[ \int_{N+v}^{N+v+1}\kappa (z)dz \right]  \left[ 2\int_{v}^{N+v}\kappa (z)dz +\int_{N+v}^{N+v+1}\kappa (z)dz \right] dv  + \\
&+& C_{\kappa} \int_{N}^{N+1} \left[ \int_{0}^{s}\kappa (z)dz \right]^2 ds \\
&\geq& C_{\kappa} \left[ \int_{0}^{N} \kappa (z)dz \right]^2 (N+1-N) \to +\infty,
\end{eqnarray*}
as $N \to +\infty$, since $\kappa \not\in L^1 (\epsilon, +\infty)$, by (\ref{chilim}). 
\end{proof}

\section{Special cases}\label{SubSec: SpecialCases}

\begin{enumerate}
\item \emph{Stable case (FBM)}: let us choose as Bernstein function the Laplace exponent of a stable subordinator, i.e. $%
\varphi(x)=x^{(1-\alpha )/2},$ for $\alpha \in (0,1)$, which satisfies conditions $\textbf{(C)-(K)}$ and corresponds to the tail L\'{e}vy measure $\nu
(s)=s^{(\alpha -1)/2}/\Gamma ((1+\alpha )/2)$. As a consequence, the process $B_{t}^{\nu }$
reduces to the fractional Brownian motion with Hurst parameter $H=\alpha /2.$
Note that, in this case, we have that $\nu \in L^{p_{1}}(0,\epsilon )\cap
L^{p_{2}}(\epsilon ,\infty )$, for $p_{1}\in \lbrack 1,2/(1-\alpha )]$ and $%
p_{2}\in (2/(1-\alpha ),\infty )$; moreover, the assumptions $\textbf{ (A1)-(A2)}$ are
satisfied for any $\alpha $ (see \cite{GRO2}).

On the other hand, for $\alpha \in (1,2)$, in order to obtain the case of the fractional Brownian motion with
Hurst parameter $H=\alpha/2\in (1/2,1)$, we must consider the function $\varphi(s)=1/\tilde{\kappa}(s)=s^{(\alpha -1)/2}$, which corresponds to the Sonine pair $\kappa (x)=x^{(\alpha
-3)/2}/\Gamma ((\alpha -1)/2)$ and $\nu (x)=x^{(1-\alpha )/2}/\Gamma ((3-\alpha )/2)$. It
is easy to check that $\textbf{(C)-(K)}$ are verified. Moreover, $\kappa$ and $\chi $
satisfy \textbf{(B1)-(B2)},
respectively; the operator $%
\mathcal{I}_{-}^{(\kappa )}$ reduces to $\mathcal{I}_{-}^{(\alpha
-1)/2} $ whose Fourier transform is given in (\ref{ex3}).

Thus, for any $\alpha \in (0,2),$ the process reduces to the FBM, represented as

\begin{equation*}
  B_{t}^{\alpha/2}=\frac{\sqrt{C_\alpha}}{\Gamma((1+\alpha)/2)} \int_{\mathbb{R}} [(t-u)^{\frac{\alpha-1}{2}}_+ -(-u)^{\frac{\alpha-1}{2}}_+]dW_u ,
\end{equation*}
where $C_\alpha =sin(\pi \alpha/2) \Gamma(\alpha+1)$ (cf. \cite{MIS}, taking into account that we restrict here to the case of positive $t$).

\item \emph{Distributed-order stable case}: the previous example can be extended both for $\alpha \in (0,1)$ and for $\alpha \in (1,2) $,
 by considering the case of a random index $\alpha$, with discrete
distribution $\left\{ {(\alpha _{k},p_{k}),k=1,...,n}\right\} $, $n\in \mathbb{N}$, where
$p_{k}=P({\alpha =\alpha _{k})}$. The corresponding Bernstein function is
given by $\varphi (x)=\sum_{k=1}^{n}p_{k}x^{(1-\alpha _{k})/2},$ for $\alpha_k \in
(0,1)$, and by $\varphi (x)=\sum_{k=1}^{n}p_{k}x^{(\alpha _{k}-1)/2},$ for $%
\alpha_k \in (1,2)$, $k=1,...,n$.

Correspondingly, the process is defined as in
Def. \ref{def4} with $\mathcal{M}_{-}^{(\varphi )}=\sum_{j=1}^{n}p_{j}%
\mathcal{D}_{-}^{(1-\alpha _{j})/2},$ for $\alpha _{j}\in (0,1)$, $%
j=1,...,n$, and $\mathcal{M}_{-}^{(\varphi )}=\sum_{j=1}^{n}p_{j}\mathcal{I}%
_{-}^{(\alpha _{j}-1)/2},$ for $\alpha _{j}\in (1,2)$, $j=1,...,n,$ by
using the functions $\nu (x)=\sum_{j=1}^{n}\frac{p_{j}}{\Gamma ((1+\alpha
_{j})/2)}x^{(\alpha _{j}-1)/2}$ and $\kappa (x)=\sum_{j=1}^{n}\frac{p_{j}}{%
\Gamma ((\alpha _{j}-1)/2)}x^{(\alpha _{j}-3)/2}$, respectively. Thus, for $\alpha_j \in (0,2),$ $j=1,...,n$, we can define the process as

\begin{equation*}
B_{t}^{\varphi }=\sum_{j=1}^{n}\frac{p_{j}\sqrt{C_{\alpha_j}}}{\Gamma ((1+\alpha
_{j})/2)}\int_{\mathbb{R}}\left[ (t-u)_{+}^{(\alpha
_{j}-1)/2}-(-u)_{+}^{(\alpha _{j}-1)/2}\right] dW_{u},
\end{equation*}%
where $C_{\alpha_j}$ can be easily evaluated.
The case of an absolutely continuous index $\alpha $, with density function $%
f_{\alpha }(\cdot ),$ can be treated similarly, by considering either $%
\varphi (x)=\int_{0}^{1}x^{(1-a)/2}f_{\alpha }(a)da$ or $\varphi
(x)=\int_{1}^{2}x^{(a-1)/2}f_{\alpha }(a)da.$

\item \emph{Tempered stable case I}: let us consider
$\varphi(x)=x(x+\lambda)^{-(\alpha +1)/2}$,
which is a Bernstein function, as far as $\alpha$ is in $(0,1)$, as can be checked by differentiating. The corresponding
tail L\'{e}vy measure is
equal to $\nu (s)=e^{-\lambda s}s^{(\alpha -1)/2}/\Gamma ((1+\alpha )/2).$ Here $%
\nu \in L^{1}(\mathbb{R}^{+})$, since, for $\epsilon >0$ and $t\geq \epsilon
,$
\begin{equation*}
\frac{1}{\Gamma ((1+\alpha )/2)}\int_{t}^{\infty }e^{-\lambda s}s^{(\alpha
-1)/2}ds=\frac{\Gamma
((\alpha +1)/2,t\lambda )}{\lambda ^{(\alpha +1)/2}\Gamma ((1+\alpha )/2)},
\end{equation*}%
where $\Gamma(\rho,x):=\int_{x}^{+\infty}e^{-z}z^{\rho -1}dz$ is the upper-incomplete Gamma function.
Correspondingly, the process $B_{t}^{\nu }$ reduces to
the tempered fractional Brownian motion with Hurst parameter $H=\alpha /2<1/2$
(see, for example, \cite{MEE}). Indeed, we have that
\begin{equation*}
B_{t}^{\nu }=\frac{\sqrt{C_{\nu }}}{\Gamma ((1+\alpha )/2)}\int_{\mathbb{R}}\left[
e^{-\lambda (t-u)_{+}}(t-u)_{+}^{(\alpha -1)/2}-e^{-\lambda
(-u)_{+}}(-u)_{+}^{(\alpha -1)/2}\right] dW_{u}.
\end{equation*}%
Then the normalizing constant can be evaluated as follows, for $t>0$,%
\begin{eqnarray*}
&&\int_{\mathbb{R}}k(t,u)^{2}du=\frac{C_{\nu }}{\Gamma ((1+\alpha )/2)^2}\int_{\mathbb{R}}\left[ e^{-\lambda
(t-u)_{+}}(t-u)_{+}^{(\alpha -1)/2}-e^{-\lambda (-u)_{+}}(-u)_{+}^{(\alpha
-1)/2}\right] ^{2}du \\
&=&\frac{C_{\nu }}{\Gamma ((1+\alpha )/2)^2}\int_{0}^{+\infty }\left[ e^{-\lambda (t+u)}(t+u)^{(\alpha
-1)/2}-e^{-\lambda u}u^{(\alpha -1)/2}\right] ^{2}du+ \\
&+&\frac{C_{\nu }}{\Gamma ((1+\alpha )/2)^2}\int_{0}^{t}e^{-2\lambda (t-u)}(t-u)^{\alpha -1}du=:[I_{\lambda ,\alpha
,t}+I_{\lambda ,\alpha ,t}^{\prime }]\frac{C_{\nu }}{\Gamma ((1+\alpha )/2)^2}.
\end{eqnarray*}%
For the first integral we have that%
\begin{eqnarray*}
I_{\lambda ,\alpha ,t} &=&\int_{0}^{+\infty }e^{-2\lambda (t+u)}(t+u)^{\alpha
-1}du+\int_{0}^{+\infty }e^{-2\lambda u}u^{\alpha -1}du+ \\
&&-2e^{-\lambda t}\int_{0}^{+\infty }e^{-2\lambda u}u^{(\alpha
-1)/2}(t+u)^{(\alpha -1)/2}du \\
&=&e^{-2\lambda t}t^{\alpha }\Psi (1;\alpha +1);2\lambda t)+(2\lambda
)^{-\alpha }\Gamma (\alpha ) -2t^{\alpha }e^{-\lambda t}\Psi ((\alpha +1)/2;\alpha +1);2\lambda t),
\end{eqnarray*}%
where $\Psi (a;c;x):=\int_{0}^{+\infty }e^{-xu}u^{a-1}(1+u)^{c-a-1}du$ is
the Tricomi's confluent hypergeometric function (see \cite{KIL}, p.30).

The second integral reads%
\begin{equation*}
I_{\lambda ,\alpha ,t}^{\prime }=\int_{0}^{t}e^{-2\lambda (t-u)}(t-u)^{\alpha
-1}du=(2\lambda )^{-\alpha }\gamma (\alpha ;2\lambda t).
\end{equation*}%
Therefore, we have that $C_{\nu }= \Gamma
((1+\alpha )/2)^2 /\left( I_{\lambda ,\alpha ,1}+I_{\lambda ,\alpha ,1}^{\prime
}\right).$
It is proved in \cite{LUC2} that the Sonine associate kernel of $\nu_{\beta, \lambda}
 (s):=e^{-\lambda s}s^{\beta -1}/\Gamma (\beta),$ for $\beta \in (0,1)$, is given by
$\kappa_{\beta, \lambda}(s)=\nu_{1-\beta, \lambda}(s)+\lambda \int_{0}^{s}\nu_{1-\beta, \lambda}(x)dx,$
thus defining the process through the fractional integral, even if possible, is rather complicated (since it involves
 integrals of incomplete gamma functions).

\item \emph{Tempered stable case II}: choosing the Bernstein function $\varphi(x)=(\lambda +x)^{\alpha }-\lambda
^{\alpha }, $ for $\lambda \geq 0$ and $\alpha \in (0,1)$, i.e. $\nu (s)=\alpha \lambda ^{\alpha }\Gamma (-\alpha ,\lambda s)/\Gamma
(1-\alpha )$, we have the
following process%
\begin{equation*}
B_{t}^{\nu }=\frac{\sqrt{C_{\nu }}\alpha \lambda ^{\alpha }}{\Gamma (1-\alpha )}%
\int_{\mathbb{R}}\left[ \Gamma (-\alpha ,\lambda (t-u)_{+})-\Gamma (-\alpha
,\lambda (-u)_{+})\right] dW_{u},
\end{equation*}%
where the normalizing constant $C_{\nu}$ can be evaluated accordingly and the assumptions $\mathbf{(C)}$ and $\mathbf{(A1)}$-$\mathbf{(A2)}$ (with $p=1$) can be easily checked by considering the properties of the upper-incomplete gamma function (see \cite{BG}).

\item \emph{Mittag-Leffler case}: if we choose the Bernstein function $\varphi(x)=\frac{x ^{1-\beta+\alpha}}{x ^\alpha +1}$, for $0<\alpha <\beta <1$, the corresponding Sonine pair is given by $\nu(x)=x^{\beta -1}E_{\alpha, \beta}(-x^\alpha)$ and $\kappa(x)=\frac{x^{\alpha -\beta}}{\Gamma (1-\beta+\alpha)}+\frac{x^{ -\beta}}{\Gamma (1-\beta)}$ (see \cite{LUC2} and \cite{LEO}, for details). It is easy to check that \textbf{(C)-(K)} hold; moreover, $\chi(x)=\frac{x^{\alpha -\beta+1}}{\Gamma (2-\beta+\alpha)}+\frac{x^{ 1-\beta}}{\Gamma (2-\beta)}$ satisfies \textbf{(B1)-(B2)}. Therefore we can define the process as
     \begin{equation*}
B_{t}^{\kappa }=\sqrt{C_{\kappa }}\int_{\mathbb{R}%
}\left[ \frac{(t-u)_{+}^{\alpha -\beta +1}-(-u)_{+}^{\alpha -\beta +1}}{\Gamma (2-\beta +\alpha )}+\frac{(t-u)_{+}^{1-\beta }-(-u)_{+}^{1 -\beta }}{\Gamma (2-\beta )}\right] dW_{u}.
\end{equation*}%
On the other hand, $\nu(\cdot)$ and $\overline{\nu}(\cdot)$  satisfy $\textbf{(A1)-(A2)}$, as can be checked, by considering the asymptotic behavior of the Mittag-Leffler function and of its derivatives (see \cite{KIL}, p.43), so that we have
\begin{equation*}
B_{t}^{\nu }=\sqrt{C_{\nu }}\int_{\mathbb{R}}\left[(t-u)_{+}^{\beta-1} E_{\alpha ,\beta} (-(t-u)_{+}^{\alpha})-(u)_{+}^{\beta-1}E_{\alpha ,\beta} (-(-u)_{+}^{\alpha})\right] dW_{u},
\end{equation*}%

\item \emph{Gamma case}: if we consider the Laplace exponent of the Gamma subordinator, i.e. $\varphi(x)=\log (1+x)$, then the L\'{e}vy density is $\overline{\nu }(s)=e^{-s}/s$, for $s >0$, and $\nu (s)=%
\func{E_1}(s)$ (where $\func{E_1}(\cdot )$ denotes the exponential integral). It is proved in \cite{COL} that
the Sonine associate kernel is $\kappa(x)=\int_{-1}^{0}\gamma(y,x)/\Gamma(y)dy$. 

It can be checked that assumptions $\mathbf{(A1)}$-$\mathbf{(A2)}$ are satisfied by $\overline{\nu}(\cdot)$ and $\nu(\cdot)$, respectively, while checking $\mathbf{(B1)}$-$\mathbf{(B2)}$ for $\kappa(\cdot)$ and $\chi(\cdot)$ is more complicated. We then define, in this case, the following process
 \begin{equation*}
B_{t}^{\nu }=\sqrt{C_{\nu }}\int_{\mathbb{R}%
}\left[ E_{1}(t-u)_{+}-E_{1}(-u)_{+}] \right] dW_{u},
\end{equation*}%
where the constant $C_{\nu }=\int_{\mathbb{R}}k(1,x)^2 dx=1.14$. Since, in this case $\nu \in L^2 (\mathbb{R}^+)$, in view of Corollary \ref{corovar}, we can establish that the process has finite variance in the long time.

\item \emph{Composite stable case}:
Let now consider the function $\kappa (\cdot )$ given in Ex.11.18 (i) in 
\cite{SCH}, i.e. $\kappa (z)=z^{-\alpha }1_{0<z\leq 1}+z^{-\beta }1_{z>
1},$ for $0<\alpha<\beta<1$, we have that $\chi(x)=\frac{x^{1-\alpha }}{1-\alpha}1_{0<x\leq 1}+[K_{\alpha,\beta}+\frac{x^{1-\beta }}{1-\beta}]1_{x>
1}$, where $K_{\alpha,\beta}:=(\alpha-\beta)/[(1-\alpha)(1-\beta)]$. Thus, in order to fulfill $\mathbf{(B2)}$, we need to add the following condition: $\alpha<2/3$. We can define the following process:

\begin{eqnarray*}
B_{t}^{\kappa }&=&\sqrt{C_{\kappa }}\int_{-\infty}^{t}\left[ \frac{(t-u)^{1-\alpha}}{1-\alpha}1_{t-1<u<t}+\left(\frac{(t-u)^{1 -\beta}}{1-\beta}+K_{\alpha,\beta}\right)1_{u<t-1}\right]dW_u+ \\
&-&\sqrt{C_{\kappa }}\int_{-\infty}^{0}\left[ \frac{(-u)^{1-\alpha}}{1-\alpha}1_{-1<u<0}+\left(\frac{(-u)^{1 -\beta}}{1-\beta}+K_{\alpha,\beta}\right)1_{u<-1}\right]dW_u.
\end{eqnarray*}
Other "composite cases" can be similarly derived from Examples \ref{Ex54} and \ref{Ex55}.
\end{enumerate}
\begin{remark}
We note that, if we consider the Bernstein function $\varphi(\theta)=\theta/(1+\theta)$ and thus $\nu (s)=e^{-s}=\overline{\nu }(s)$, we have
\begin{equation*}
B_{t}^{\nu }=\sqrt{C_{\nu }}\int_{\mathbb{R}}\left[ e^{-(t-u)_{+}}-e^{-(-u)_{+}}%
\right] dW_{u},
\end{equation*}%
with $C_{\nu }=e/(e-1)$. Its covariance function is then equal to
\begin{equation*}
\mathrm{Cov}\left[ B_{t}^{\nu },B_{s}^{\nu }\right] =C_{\nu }\left(
1-e^{-|t|}-e^{-|s|}+e^{-|t-s|}\right),
\end{equation*}%
for $t \in \mathbb{R}$, and $B_{t}^{\nu} $ is an Ornstein-Uhlenbeck type process with starting point in $0$.
However, we have not included this case in our examples above, because the L\'{e}vy measure $\overline{\nu }(\cdot)$ is finite on $[0,+\infty)$
and thus the kernel of the derivative in not singular. Therefore the condition $\textbf{(C)}$ is violated.
\end{remark}

\begin{appendix}

\section{Fractional derivatives and integrals}
We recall some well-known definitions of fractional operators. 
For a function $f \in L^p (\mathbb{R}),$ $\alpha \in (0,1)$  and $p \in [1,1/\alpha)$ the \emph{right sided} and \emph{left-sided Riemann-Liouville fractional integrals}
are defined, respectively, as follows:
\begin{equation}
(\mathcal{I}_{-}^{\alpha }f)(x):=\frac{1}{\Gamma (\alpha )}%
\int_{x}^{+\infty }f(t)(t-x)^{\alpha-1 }dt  \label{i1}
\end{equation}%
and%
\begin{equation}
(\mathcal{I}_{+}^{\alpha }f)(x):=\frac{1}{\Gamma (\alpha )}%
\int_{- \infty}^{x }f(t)(x-t)^{\alpha-1 }dt,  \label{i2}
\end{equation}%
for $%
x\in \mathbb{R}$.

Under the same assumptions on $\alpha,p$ and for $f \in \mathcal{I}_{\pm}^{\alpha }(L^p (\mathbb{R})),$ the so-called \emph{right-sided} and \emph{left-sided
Riemann-Liouville fractional derivative}, are given respectively by

\begin{equation}
(\mathcal{D}_{-}^{\alpha }f)(x):=-\frac{1}{\Gamma (1-\alpha )}\frac{d}{dx}%
\int_{x}^{+\infty }f(t)(t-x)^{-\alpha }dt  \label{rl}
\end{equation}%
and%
 \begin{equation}
        (\mathcal{D}_{+}^{\alpha }f)(x):=\frac{1}{\Gamma (1-\alpha )}\frac{d}{dx}%
        \int_{-\infty}^{x }f(t)(x-t)^{-\alpha }dt,  \label{rl2}
    \end{equation}%
for $%
x\in \mathbb{R}$.
Analogously, the \emph{right-sided} and \emph{left-sided Weyl fractional derivative}, for $f\in L^{p}(\mathbb{R})$ and $%
    x\in \mathbb{R}$, are defined, respectively, as
      \begin{equation}
(\mathfrak{D}_{-}^{\alpha }f)(x)=\frac{\alpha }{\Gamma (1-\alpha )}%
\int_{x}^{+\infty }\left[ f(t)-f(x)\right] (t-x)^{-\alpha -1}dt {\label{wr}}
\end{equation}%
 and%
    \begin{equation}
        (\mathfrak{D}_{+}^{\alpha }f)(x)=\frac{\alpha }{\Gamma (1-\alpha )}%
        \int_{-\infty}^{x}\left[ f(t)-f(x)\right] (x-t)^{-\alpha -1}dt {\label{wr2}}
    \end{equation}%
(see \cite{MIS}, p.3 and 4).

\section{Distributions' space, S-transform and characterization's theorems}

    \subsection{Hida distribution space and its extension}\label{App:HidaDistribution}
        We define $(L^2):=L^2(\mathcal{S}',\mathcal{B},\mu)$ where $(\mathcal{S}',\mathcal{B},\mu)$ is the white noise space  and recall that each element
        of the Hilbert space $(L^2)$ can be uniquely expanded into a series of multiple
        Wiener integrals by the Wiener-Ito-Segal chaos decomposition; see \cite{HID}.
        Indeed, each element $F \in (L^2)$ corresponds to a sequence of functions,
        $\{f^{(n)} \in L_{\mathbb{C}}^{2 \; \hat{\otimes} n}$, $n \in \mathbb{N}\}$,
        such that $F(\cdot)=\sum_{n=0}^{\infty} \langle : \cdot^{\otimes n} : , f^{(n)} \rangle$ in
        the sense of $(L^2)$, where $\langle : \cdot^{\otimes n} : , f^{(n)} \rangle$ is the Hermite
        polynomial in $\langle \cdot , f \rangle$ of order $n$. Furthermore, we have that  $\vvvert  F \vvvert_{(L^2)}^2=\sum_{n=0}^\infty n! \|f^{(n)}\|^2_{L^2(\R^n)}$.
        \\
        In Gaussian analysis, many authors built Gelfand triples starting from $(L^2)$ or from the space of linear combination of polynomials
        \[ \mathcal{P}(\mathcal{S}'):= \{ \varphi \in (L^2) \mid \varphi(\omega)=\sum_{n=0}^N  \langle : \omega^{\otimes n} : , f^{(n)} \rangle, \, f^{(n)}\in L_{\mathbb{C}}^{2 \; \hat{\otimes} n}, \, N \in \mathbb{N}   \};  \]
        see \cite{HID}, \cite{KON} and \cite{KON2}.
        Here we present briefly the construction of the triple $(\mathcal{S})\subset (L^2)\subset (\mathcal{S})'$,
        where $(\mathcal{S})'$ is the Hida distributions space and $(\mathcal{S})$ is the test functions' space. These spaces were extended by \cite{KON} as follows.
        \\
       The Hida space $(\mathcal{S})'$ is defined by using the family of norms:

        \[  \vvvert F \vvvert_p:=\left( \sum_{n \geq 0} n!\|f^{(n)}\|^2_p \right)^{1/2},\quad p \in \mathbb{N} \text{ and } F \in (L^2) \]
        where $\|f^{(n)}\|_p:=\|(A^p)^{\otimes n} f^{(n)}\|_{L^2(\mathbb{R}^n)}$, with $f^{(1)}=f,$ and $A$ is the harmonic oscillator (see Section V.3 in \cite{REE}). By defining

        \[ (\mathcal{S})_p:=\{ \eta \in (L^2) \mid \vvvert \eta\vvvert_p < \infty \}  \]
        and $(\mathcal{S})_{-p}$ as its dual with respect to $(L^2)$, we introduce the following Gelfand triple:

        \[  (\mathcal{S}) \subset (L^2) \subset (\mathcal{S})'  \]
        where the space of test functions $(\mathcal{S})$ can be represented as

        \[   (\mathcal{S})= \underset{p \in \mathbb{N}}{\text{pr lim}}\,(\mathcal{S})_p .  \]
       Moreover, let the space of Hida distributions $(\mathcal{S})'$ be defined as

        \[   (\mathcal{S})'= \underset{p \in \mathbb{N}}{\text{ind lim}}\,(\mathcal{S})_{-p},  \]
        where $\text{pr lim}$ and $\text{ind lim}$ stand for the projective limit and the inductive limit, respectively (see pag.~74 in \cite{HID}).
        The Hida test functions space $(\mathcal{S})$ and its dual $(\mathcal{S})'$ can be extended by noting that the space of linear combination of polynomials $\mathcal{P}(\mathcal{S}')$ is dense in $(L^2)$. Moreover, its completion, with respect to the family of norms

        \[ \vvvert \varphi \vvvert_{p,q,\beta}:= \sum_{n \geq 0} (n!)^{1+\beta} 2^{nq} \|\varphi^n\|^2_p, \quad p,q \in \mathbb{N}, \beta \in [0,1],\]
        coincides with the set of Hilbert spaces $(L^2_p)_q^{\beta}$. By using their projective and inductive limits, we introduce new spaces of test functions and distributions, which extend the triple formed by the Hida space, $(L^2)$ and the distributions' space. Indeed, we obtain the space of test functions by

        \[ (\mathcal{S})^{\beta}:= \underset{p,q \in \mathbb{N}}{\text{pr lim}}\, (L^2_p)_q^{\beta} .\]

        Moreover, denoting the dual space of $(L^2_p)_q^{\beta}$ with respect to $(L^2)$ by $(L^2_{-p})_{-q}^{-\beta}$, we endow it with the family of norms

        \[  \vvvert \Phi \vvvert_{-p,-q,-\beta}:= \sum_{n \geq 0} (n!)^{1-\beta} 2^{-nq} \|\Phi^n\|^2_{-p}, \quad p,q \in \mathbb{N}, \beta \in [0,1], \]
        such that it is formed by those distributions $\Phi$ for which $\vvvert \Phi \vvvert_{-p,-q,-\beta}<\infty$.
        Thus, we gain the distributions' space $(\mathcal{S})^{-\beta}$ by

        \[  (\mathcal{S})^{-\beta}= \underset{p,q \in \mathbb{N}}{\text{ind lim}}\,(L^2_{-p})_{-q}^{-\beta}.   \]

        Hence, we have that the triple $(\mathcal{S})^{\beta} \subset (L^2) \subset (\mathcal{S})^{-\beta}$ extends $(\mathcal{S}) \subset (L^2) \subset (\mathcal{S})'$ for $0\leq \beta\leq 1$:

        \[   (\mathcal{S})^1\subset(\mathcal{S})^{\beta}\subset(\mathcal{S})\subset(L^2)\subset(\mathcal{S})'\subset(\mathcal{S})^{-\beta}\subset(\mathcal{S})^{-1},\]
        where we identify $(\mathcal{S})^0 \equiv (\mathcal{S})$; see \cite{KON}.

        \subsection{S-transform and characterization theorems}
        We now introduce an integral transform, called $S$-transform, in order to characterize the distributions space $(\mathcal{S})^{-1}$ for calculations.
        For further insight on the $S$-transform and the distributions' spaces $(\mathcal{S})^{-\beta}$, $\beta \in [0,1]$; see
        \cite{KON}.

    \begin{definition}\label{def:STransform}
        Let $U_{p,q}:=\{\xi \in \mathcal{S}_{\mathbb{C}} \mid 2^q \|\xi\|_p^2<1 \}$, for some $p,q \in \mathbb{N}$. For any $\Phi \in (\mathcal{S})^{-1}$ we define the $S$-transform as

        \[ (S \Phi )(\xi):= e^{-\frac{1}{2}\langle \xi, \xi \rangle } \int_{S'}e^{\langle \omega, \xi\rangle}\Phi(\omega)\mu(d\omega), \quad \xi \in  U_{p,q}. \]
    \end{definition}

The theorem below states the isomorphism relation between the distributions space $(\mathcal{S})^{-1}$ and the space of germs of functions holomorphic at zero $\text{Hol}_{0}(\mathcal{S}_{\mathbb{C}})$.

    \begin{theorem}[cf.~Thm.~8.34 in \cite{KON2}]\label{thm:STransfoIsomo}
        The $S$-transform is a topological isomorphism from $(\mathcal{S})^{-1}$ to $\text{Hol}_{0}(\mathcal{S}_{\mathbb{C}})$.
    \end{theorem}


   \begin{theorem}[cf.~Thm.~2.12 in \cite{GRO2}]\label{thm:CharctConvergSequenceHidaSpace}
        Let $\{\Phi_n\}_{n\in \mathbb{N}}$ be a sequence in $(\mathcal{S})^{-1}$. Then $\{\Phi_n\}_{n \in \mathbb{N}}$
        converges strongly in $(\mathcal{S})^{-1}$ if and only if there exist $p,q \in \mathbb{N}$ such that:
        \begin{itemize}
            \item[i)] For each $\xi \in U_{p,q}$ the sequence $\{(S\Phi_n) (\xi)\}_{n \in \mathbb{N}}$ is a Cauchy sequence;
            \item[ii)] Every $S(\Phi_n)(\cdot)$ is holomorphic on $U_{p,q}$, and there exist a constant $C>0$ such that
            \[ |S(\Phi_n) (\xi)|\leq C  \]
            for all $\xi \in U_{p,q}=\{\xi \in \mathcal{S}_{\mathbb{C}} \mid 2^q \|\xi\|_p^2<1 \}$ and for all $n \in \mathbb{N}$.
        \end{itemize}
    \end{theorem}
\begin{remark}
The sufficient conditions of the theorem above were also stated in Theorem 5 in \cite{KON}.
\end{remark}
\end{appendix}

\end{document}